\documentclass[12pt]{amsart}

\usepackage{amssymb}
\usepackage{enumerate}
\usepackage{cite}
\newtheorem{theorem}            {Theorem}[section]
\newtheorem{corollary}          [theorem]{Corollary}
\newtheorem{proposition}        [theorem]{Proposition}
\newtheorem{definition}         [theorem]{Definition}
\newtheorem{lemma}              [theorem]{Lemma}

\newtheorem{example}            [theorem]{Example}
\newtheorem{remark}             [theorem]{Remark}

\def\Zp{\mathcal Z_p}
\def\tV{\tilde{\V}}

\def\tQ{\tilde Q}
\def\tP{\tilde P}
\def\tK{\tilde{K}}

\def\tS{\tilde{S}}
\def\tT{\tilde{T}}

\def\cK{\mathcal K}

\newcommand{\ts}[1]{\frac{1}{\overline{#1}}}

\def\V{\mathcal V}
\def\N{\mathcal N}

\def\zetastar{\overline{\zeta}}
\def\etastar{\overline{\eta}}
\def\lambdastar{\overline{\lambda}}
\def\mustar{\overline{\mu}}

\def\mQ{K_Q}
\def\cQ{\mathcal Q}
\def\tcQ{\tilde{\mathcal Q}}
\def\tmQ{\tilde{\mQ}}

\def\oz{\overline{z}}
\def\ow{\overline{w}}
\def\ozeta{\overline{\zeta}}
\def\oeta{\overline{\eta}}

\begin{document}
\title[Pick interpolation on distinguished varieties]{Nevanlinna-Pick interpolation on distinguished varieties in the bidisk}

\author[M.~T.~Jury, G.~Knese,  and
S.~McCullough]{Michael T. Jury$^1$, Greg Knese$^2$, 
  and Scott McCullough$^3$}

\address{Department of Mathematics\\
  University of Florida\\
  Box 118105\\
  Gainesville, FL 32611-8105\\
  USA}

\email{mjury@ufl.edu}

\address{Department of Mathematics\\
  University of Alabama\\
  Tuscaloosa, AL 35487-0350}

\email{geknese@bama.ua.edu}

\address{Department of Mathematics\\
  University of Florida\\
  Box 118105\\
  Gainesville, FL 32611-8105\\
  USA}

\email{sam@ufl.edu}

\subjclass[2000]{}

 
\thanks{Thanks to BIRS for providing a useful meeting place. \\
  ${}^1$Research supported by NSF grants DMS 0701268 and DMS 1101461.
  \\ ${}^2$ Research supported by NSF grant DMS 1048775. Research
  initiated by first two authors at the Fields Institute.  Travel to
  UF supported in part by UF. \\ ${}^3$ Research supported by NSF
  grants DMS 0758306 and DMS 1101137.}

\date{\today}

\begin{abstract}
This article treats Nevanlinna-Pick interpolation in the setting of a special
class of algebraic curves called distinguished varieties.  An
interpolation theorem, along with additional operator theoretic
results, is given using a family of reproducing kernels
 naturally associated to the variety.  The 
 examples of the  Neil parabola and doubly connected domains
 are discussed. 
\end{abstract}

 \maketitle
%{\huge \today}
%\tableofcontents

\section{Introduction}
  \label{sec:intro}
 Versions of the Nevanlinna-Pick interpolation theorem stated in terms
 of the positivity of a family of Pick matrices have a long tradition
 beginning with Abrahamse's interpolation Theorem on multiply
 connected domains \cite{Ab}.  The list \cite{AD, Ag, AMpolydisc,
   AMnp, AMbook, AMint, B, BB, BBtH, BCV, BTV, CLW, DH, DPRS, DP, FF, H, MS,
   Paulsen, RIEOT, R2, Sarason} is just a sample of now classic
 papers and newer results in this direction related to the present
 paper.  Here we consider Pick interpolation on a distinguished
 variety.  For general facts about distinguished varieties we have
 borrowed heavily from \cite{AMS, AMvarieties, Kn1, Kn2}.

The classical Nevanlinna-Pick interpolation theorem says that, given
points $z_1, \dots z_n$ in the unit disk $\mathbb D\subset \mathbb C$
and points $\lambda_1, \dots \lambda_n$ in $\mathbb D$, there exists a
holomorphic function $f:\mathbb D\to \mathbb D$ with
$f(z_i)=\lambda_i$ for each $i$ if and only if the $n\times n$ {\em
  Pick matrix}
\begin{equation}\label{eqn:pick_classical}
\left(\frac{1-\lambda_i\overline{\lambda_j}}{1-z_i\overline{z_j}} \right)_{i,j=1}^n
\end{equation}
is positive semidefinite.  From the modern point of view (that is,
the point of view of ``function-theoretic operator theory''), one
interprets this condition as checking the positivity of
$(1-\lambda_i\overline{\lambda_j})$ against the {\em Szeg\H{o} kernel}
$k(z_i,z_j):=(1-z_i\overline{z_j})^{-1}$ on the interpolation nodes
$z_1, \dots z_n$.  The Szeg\H{o} kernel $k(z,w)$ is the reproducing
kernel for the Hardy space $H^2(\mathbb D)$, which is a Hilbert space
of holomorphic functions on $\mathbb D$.  Thus, this point of view
repackages the constraint $f:\mathbb D\to \mathbb D$ (that is, $|f|$
is bounded by $1$ in $\mathbb D$) as the condition that $f$ multiply
$H^2(\mathbb D)$ into itself contractively.  See \cite{AMbook} for an
extended exposition of this point of view.

One strength of this perspective is that gives a natural framework in
which to pose and solve interpolation problems on other domains (in
$\mathbb C$ or $\mathbb C^n$).  For example, Abrahamse \cite{Ab}
considered the analogous interpolation problem in a $g$-holed planar
domain $R$.  His theorem says, for a canonical family of reproducing
 kernels $k_t$ on $R$ naturally parametrized  by the
 $g$-torus $\mathbb T^g$,  that given $z_j$'s in $R$ and
$\lambda_j$'s in $\mathbb D$, there exists a holomorphic interpolating
function $f:R\to \mathbb D$ if and only if each of the family of Pick
matrices
\begin{equation}\label{eqn:abrahamse_family}
[(1-\lambda_i\overline{\lambda_j})k_t(z_i,z_j)]_{i,j=1}^n
\end{equation}
is positive semidefinite.  
%Here $k_t$ ranges over an explicitly
%described family of reproducing kernels on $R$, naturally parametrized
%by the $g$-torus $\mathbb T^g$.  
More recently, Davidson, Paulsen,
Raghupathi and Singh \cite{DPRS} considered the original Pick problem
on the disk, but with the additional constraint that $f^\prime(0)=0$.
Again, positivity against a particular family of kernels is necessary
and sufficient.  A very general approach to interpolation via kernel
families may be found in \cite{JKM-kernels}.

On the other hand, {\em distinguished varieties} have recently emerged
as a new venue in which to investigate function-theoretic operator
theory \cite{AMS, AMvarieties, Kn1, Kn2}.  By one definition, a
distinguished variety is an algebraic variety $\mathcal Z\subset
\mathbb C^2$ with the property that if $(z,w)\in \mathcal Z$, then
$|z|$ and $|w|$ are either both less than $1$, both greater than $1$,
or both equal to $1$.  (We shall usually consider only the intersection
$\mathcal V=\mathcal Z\cap \mathbb D^2$, and by abuse of language
refer to this as a distinguished variety as well.)  The simplest
non-trivial example is the {\em Neil parabola} $\mathcal N=
\{(z,w):z^3=w^2\}$.  Just as von Neumann's inequality and the
Sz.-Nagy-Foias dilation theorem establish a connection between
contractive operators on Hilbert space and function theory in the unit
disk $\mathbb D$, the function theory on a distinguished variety is
linked to the study of pairs of commuting contractions $S,T$ on
Hilbert space which obey a polynomial relation $p(S,T)=0$.

The purpose of this paper is to prove a Pick interpolation theorem for
bounded analytic functions on distinguished varieties.  The main
theorem identifies a canonical collection of kernels over the variety.
Each kernel corresponds to commuting pair of isometries with finite
rank defect and Taylor spectrum in the closure of the variety.  It
turns out, in fact, that two of the examples mentioned above
(constrained interpolation in the disk, and interpolation on multiply
connected domains) can be recast as interpolation problems on
distinguished varieties.  In addition to the interpolation theorem,
this article contains information about such pairs of isometries,
including a geometric picture of the (minimal) unitary extension with
spectrum in the intersection of the boundary of the bidisk and the
closure of the distinguished variety.  We also prove a polynomial
approximation theorem for bounded analytic functions on varieties.
The remainder of this introductory section provides some background
material and states the main results more fully.

\subsection{Distinguished Varieties}
  A subset  $\V$ of $\mathbb D^2$ is a {\em distinguished variety} if  
  there exists a square free polynomial $p \in \mathbb C[z,w]$ such that 
\begin{equation*}
  \V= \mathcal Z_p \cap \mathbb D^2
\end{equation*}
and
\begin{equation*}
\mathcal Z_p \subset \mathbb D^2 \cup \mathbb T^2 \cup \mathbb E^2.
\end{equation*}
Here $\mathcal Z_p$ is the zero set of $p$;
\begin{equation*}
 \mathbb D^2=\{(z,w)\in\mathbb C: |z|,\ |w| <1\}
\end{equation*}
is the bidisk;   $\mathbb T^2=\{(z,w)\in\mathbb C^2: |z|=1=|w|\}$
is the distinguished boundary of the bidisk; and 
\begin{equation*}
  \mathbb E^2=\{(z,w)\in\mathbb C^2: |z|,\ |w|>1\}
\end{equation*} 
denotes the {\it exterior bidisk}.  An alternate, but equivalent,
definition of distinguished variety is an algebraic set in the bidisk
that exits through the distinguished boundary $\mathbb{T}^2=(\partial
\mathbb{D})^2$ (see \cite{AMvarieties}, \cite{Kn1}).

The polynomial $p$ can be chosen to have the symmetry
\begin{equation}
 \label{eq:Zp-symmetry}
   p(z,w)=z^n w^m \overline{p(\frac{1}{\oz},\frac{1}{\ow})}
\end{equation}
 and hence $\mathcal Z_p$ is invariant under the 
 map $(z,w)\to (\frac{1}{\oz},\frac{1}{\ow})$ \cite{Kn1}.
  
 Write $(n,m)$ for the bidegree of $p$; i.e. $p$ has degree $n$ in $z$
 and $m$ in $w$.  By a fundamental result of Agler and McCarthy
 \cite{AMvarieties}, $\V$ admits a determinantal representation
\begin{equation}
 \label{V-Phi}
\V =\{(z, w)\in\mathbb D^2 : \det (wI_m -\Phi(z))=0\}
\end{equation}
where $\Phi$ is an $m\times m$ rational matrix function which is 
analytic on the closed disk $\overline{\mathbb D}$ and unitary 
on $\partial \mathbb D$. 

 Given such a $\Phi$ there 
exist (row) vector-valued polynomials
\begin{equation*}
 Q(z, w) =(q_1(z,w) \dots q_m(z,w)), \quad P(z, w) 
   =(p_1(z,w) \dots p_n(z,w)), 
\end{equation*}
 such that
\begin{equation}\label{eqn:Phi_eigen}
 \Phi(z)^*Q(z,w)^* =\overline w Q(z,w)^* 
\end{equation}
and
\begin{equation}\label{E:pq}
(1-w\overline{\eta}) Q(z, w) Q(\zeta, \eta)^* =(1-z\overline{\zeta}) 
  P(z, w) P(\zeta,\eta)^* 
\end{equation}
for all $(z,w)$ and $(\zeta, \eta)$ in $\V$.  In fact, such
polynomials can be chosen such that $Q$ has degree at most $m-1$ in
$w$ and $P$ has degree at most $n-1$ in $z$ \cite{Kn1}.  Moreover,
every pair $P,Q$ of polynomial vector functions satisfying
(\ref{E:pq}) on $\V$ arises in this way; i.e., there is a rational
$\Phi$ such that \eqref{V-Phi} and \eqref{eqn:Phi_eigen} hold.  This
last assertion is a consequence of Lemma \ref{lem:PhiQP} below.

%%%?? Greg--do you know a reference? ??
%% Since we prove this, I think there is no need to reference me.
%% I'm don't think I prove this exactly anyway. --Greg

A pair $P$ and $Q$ satisfying \eqref{E:pq} 
 determines the positive definite kernel
 $K:\V\times \V\to \mathbb C$, 
\begin{equation*}
  K(z,w)=\frac{Q(z,w)Q(\zeta,\eta)^*}{1-z\overline{\zeta}}
     = \frac{P(z,w)P(\zeta,\eta)^*}{1-w\overline{\eta}}
\end{equation*}
 on $\V\times \V$. 
 %Thus, $K:\V\times \V\to \mathbb C$ is a positive semi-definite function.

 It is natural to generalize the construction of the
 kernel $K$ , allowing for matrix-valued $P$ and $Q$. Let $M_{\alpha,\beta}$ denote 
 the set of $\alpha\times \beta$ matrices with entries from $\mathbb C$.

\begin{definition}\rm
 \label{def:admissablePQ}
  A {\it rank $\alpha$ admissible pair},
  synonymously a {\it $\alpha$-admissible pair}, 
  is a pair $(P,Q)$ 
  of matrix polynomials $P,Q$ in two variables 
  such that,
 \begin{itemize}
  \item[(i)] $Q(z,w)$ is  $M_{\alpha,m\alpha}$-valued
    and $P(z,w)$ is  $M_{\alpha,n\alpha}$-valued; 
%%      are matrix-valued polynomials
 %with $Q$ having
 %   degree at most $m-1$ in $w$ and $P$ having degree
 %   at most $n-1$ in $z$;
  \item[(ii)] both $Q(z,w)$ and $P(z,w)$
      have rank $\alpha$ (that is, full rank) at some point of each irreducible
      component of $\mathcal Z_p$;  and
  \item[(iii)] for $(z,w),(\zeta,\eta)\in \Zp$, 
   \begin{equation}\label{eqn:admissible_pair_def}
     \frac{Q(z, w) Q(\zeta,\eta)^*}{(1-z\overline{\zeta})}=
       \frac{P(z, w) P(\zeta,\eta)^*}{(1-w\overline{\eta})}.
   \end{equation} 
 \end{itemize}
\noindent   A pair $(P,Q)$ is {\it admissible} if it is $\alpha$-admissible
   for some $\alpha$.
\end{definition}

\begin{remark}\rm
Though both sides of (\ref{eqn:admissible_pair_def}) define
meromorphic functions on $\mathbb C^2\times \mathbb C^2$, we emphasize
that the equality is assumed to hold only on $\Zp\times \Zp$.
\end{remark}

\iffalse
%%% Maybe leave this remark out unless it gets used later
\begin{remark} \rm
  An equivalent theory emerges if, in item (i),  
  $Q$ is assumed to be a polynomial of degree
  at most $m-1$ in $w$ and rational in $z$; and 
  likewise $P$ is a polynomial of degree at
  most $n-1$ in $z$ and is rational in $w$.
  Note, in this case, the poles of $Q(z,w)$ 
  and $P(z,w)$   
  lie outside $\mathbb D^2\cup \mathbb T^2 \cup \mathbb E^2$.
\end{remark}
\fi

  Let $\tV$ denote the intersection of $\Zp$
  with the exterior bidisk $\mathbb E^2$.  
  In view of equation \eqref{eq:Zp-symmetry}, 
  $\tV=\{(\frac{1}{\oz},\frac{1}{\ow}):(z,w)\in \V\}$.

\begin{definition}\rm
  \label{def:Ks}
  A rank $\alpha$ admissible pair $(P,Q)$ determines (positive semidefinite)
  kernels $K:\V\times \V\to M_\alpha$ and $\tK:\tV\times \tV\to M_\alpha$,
 \begin{equation*}
  \begin{split}
    K((z,w),(\zeta,\eta)) = &\frac{Q(z,w)Q(\zeta,\eta)^*}{1-z\overline{\zeta}}
                          = \frac{P(z,w)P(\zeta,\eta)^*}{1-w\overline{\eta}} \\
    \tK((z,w),(\zeta,\eta)) = &\frac{Q(z,w)Q(\zeta,\eta)^*}{z\overline{\zeta}-1}
                             = \frac{P(z,w)P(\zeta,\eta)^*}{w\overline{\eta}-1},
  \end{split}
 \end{equation*}
   which we call an {\it admissible pair} of kernels.

\begin{remark} \rm
 \label{rem:tK}
  The kernel  $\tK$ can also be written as
 \begin{equation*}
   \tK((z,w),(\zeta,\eta)) = \frac{1}{z\overline{\zeta}} 
      \frac{Q(z,w) Q(\zeta,\eta)^*}{1-\frac{1}{z}\frac{1}{\overline{\zeta}}}
       = \frac{1}{w\overline{\eta}} \frac{P(z,w)P(\zeta,\eta)^*}
              {1-\frac{1}{w}\frac{1}{\overline{\eta}}}.
 \end{equation*}
\end{remark}
   
   The corresponding reproducing Hilbert spaces are denoted $H^2(K)$
   and $H^2(\tK)$.  The operators $S=M_z$ and $T=M_w$ of
   multiplication by $z$ and $w$ respectively are contractions on
   $H^2(K)$.  Likewise the operators $\tS$ and $\tT$ of multiplication
   by $\frac{1}{z}$ and $\frac{1}{w}$ respectively are contractions on
   $H^2(\tK)$.  We will see later that they are in fact isometric.
\end{definition}

These pairs of operators $(S,T)$ play the role of {\it bundle shifts}
\cite{AD,Sarason} over $\V$, terminology which is explained by
Theorems \ref{thm:isometries} and Theorem
\ref{thm:dilation-with-relations} below.  In addition to these
theorems, a main result of this paper is a version of Pick
interpolation for $\V$ - Theorem \ref{thm:pick-final}.  In the next
subsection we describe these results more fully.

%  The remainder of this introduction is organized as follows. 
%  The main results is stated in  subsection \ref{subsec:main}. 
%  Examples and a comparison of the interpolation
%  result here with interpolation in subalgebras
%  of the disc algebra as found in \cite{DPRS} and on
%  multiply connected domains is the content of
%  subsection \ref{subsec:examples}; 
%  and finally a guide to the remainder of the paper
%  in subsection \ref{subsec:guide}.

\subsection{Main Results}
 \label{subsec:main}
The main result of this paper is Theorem~\ref{thm:pick-final}, the Pick interpolation theorem on distinguished varieties.  It is based on three subsidiary results, each of some interest in its own right.  
\begin{definition}
  Let $\V$ be a variety.  A function $f:\V\to \mathbb C$ is {\em holomorphic} at a point $(z,w)\in \V$ if there exists an open set $U\subset \mathbb C^2$ containing $(z,w)$ and a holomorphic function $F:U\to\mathbb C$ such that $F$ agrees with $f$ on $\V\cap U$.  
\end{definition}

If $(z,w)$ is a smooth point of $\mathcal V$, then the (holomorphic) implicit function theorem tells us there is a local coordinate in a neighborhood $O$ of $(z,w)$ in $\mathcal V$ making this neighborhood into a Riemann surface; a function is holomorphic at $(z,w)$ by the above definition if and only if it is holomorphic as a function of the local coordinate.  The importance of the definition, therefore, is that it allows us to make sense of holomorphicity near singular points of $\mathcal V$.  See \cite[Chapter 3]{TaylorSCV}.

We let $H^\infty(\V)$ denote the set of functions that are bounded on $\V$ and holomorphic at every point of $\V$.  It is a Banach algebra under the supremum norm 
\begin{equation}\label{eqn:sup-norm-def}
  \|f\|_\infty :=\sup_{(z,w)\in\V}|f(z,w)|.
\end{equation}
(That $H^\infty(\V)$ is complete follows from the non-trivial fact that a locally uniform limit of functions holomorphic on $\V$ is holomorphic on $\V$, see \cite{TaylorSCV} Theorem 11.2.5.)

\begin{theorem}\label{thm:pick-final}
   Let $(z_1,w_1),\dots,(z_n,w_n)\in \V$ and 
  $\lambda_1,\dots,\lambda_n \in\mathbb D$ be given.
  There exists an $f\in H^\infty(\V)$ such that $\|f\|_\infty\leq 1$ and
 \begin{equation*}
   f(z_j,w_j)=\lambda_j \text{ for each $j=1, \dots n$}
 \end{equation*}
  if and only if
 \begin{equation*}
  (1-\lambda_\ell\overline{\lambda_j}) K((z_j,w_j),(\zeta_\ell,\eta_\ell))
 \end{equation*}
is positive semi-definite for every admissible kernel $K$.
\end{theorem}

The proof of this theorem splits into two parts.  We first introduce
an auxiliary algebra $H^\infty_{\cK}(\V)$ of bounded analytic functions
on $\V$ (but equipped with a norm $\|\cdot \|_{\V}$ that is defined
differently than the supremum norm) and prove that the condition of
Theorem~\ref{thm:pick-final} is necessary and sufficient for
interpolation in this algebra.  The second part of the proof consists
in showing that in fact $H^\infty_{\cK}(\V) =H^\infty(\V)$
isometrically ({\em a priori} it is obvious only that
$H^\infty_{\cK}(\V)\subseteq H^\infty(\V)$ contractively).  This second
part in turn splits in two: we first prove a unitary dilation theorem
for algebraic pairs of isometries; as a corollary we find that
$H^\infty_{\cK}$ contains all polynomials $q$, and
$\|q\|_{\V}=\|q\|_\infty$.  Finally we prove a polynomial
approximation result for $H^\infty(\V)$ (Theorem~\ref{thm:poly-approx}) which allows us to extend this
isometry to all of $H^\infty(\V)$.

%contains three main results.  The first is a Pick
%   interpolation theorem in a class of functions $H^\infty_{\mathcal K}(\V)$ on
%   $\V$.  The second two concern the pair of operators of
%   multiplication by $z$ and $w$ on $H^2(K)$ for an admissible kernel
%   $K$.  

\subsubsection{Interpolation in $H^\infty_{\mathcal K}(\V)$}
 \label{subsubsec:H-infty}
    We now define the auxiliary algebra $H^\infty_{\mathcal K}(\V)$ in which we will interpolate.
    Let $W$ be a set,  $n$ a positive integer, and let $M_n$
    denote the $n\times n$ matrices with entries from $\mathbb C$. 
     A function $f:W\times W\to M_n$ is positive semi-definite
    if, for each finite subset $F\subset W$, the $n|F|\times n|F|$ matrix
  \[
    \begin{pmatrix} f(u,v)\end{pmatrix}_{u,v\in F}
  \]
    is positive semi-definite; we write $f(u,v)\succeq 0$.
\begin{definition}\rm
   \label{defn:schur_class}
  Say a function $f:\V\to\mathbb C$ belongs to $H^\infty_{\mathcal K}(\V)$ if there
  exists a real number $M>0$ such that 
 \begin{equation}\label{eqn:H_infty_defn}
    (M^2-f(z,w)f(\zeta,\eta))K((z,w),(\zeta,\eta)) \succeq 0
 \end{equation}
  for all admissible kernels $K$.  The norm $\|f\|_\V$ is defined to
  be the infimum of all $M$ such that (\ref{eqn:H_infty_defn}) holds
  for all admissible $K$.
\end{definition}
In other words, $\|f\|_{\V}\leq M$ if and only if for each admissible
$K$, the operator $M_f$ of multiplication by $f$ is bounded on
$H^2(K)$, with operator norm at most $M$; and thus
\begin{equation}\label{eqn:V_norm_is_mult_norm}
\|f\|_{\V}=\sup_K {\|M_f\|}_{B(H^2(K))}.
\end{equation}

  We are finally ready to state the interpolation theorem.
\begin{theorem}
 \label{thm:interpolate}
   Let $(z_1,w_1),\dots,(z_n,w_n)\in \V$ and 
  $\lambda_1,\dots,\lambda_n \in\mathbb D$ be given.
  There exists an $f\in H^\infty_{\mathcal K}(\V)$ such that $\|f\|_\V\leq 1$ and
 \begin{equation*}
   f(z_j,w_j)=\lambda_j \text{ for each $j=1, \dots n$}
 \end{equation*}
  if and only if
 \begin{equation*}
  (1-\lambda_\ell\overline{\lambda_j}) K((z_j,w_j),(\zeta_\ell,\eta_\ell))
     \succeq 0
 \end{equation*}
  for every admissible kernel $K$.
\end{theorem}

 Theorem \ref{thm:interpolate} is proved in Section
 \ref{sec:kernel-structures}.

 The problem of extending a function defined on
 $\V$ to all of the bidisk is treated in
 \cite{AMint}.

\subsubsection{Commuting isometries with spectrum in $\overline{\V}$}

\begin{theorem}
 \label{thm:isometries}
Let $K$ be an admissible kernel and write 
\begin{equation*}
  S=M_z, \quad T=M_w
\end{equation*} 
for the coordinate multiplication operators on $H^2(K)$.  Then:
\begin{itemize}
\item[(i)] $S$ and $T$ are pure commuting isometries,
\item[(ii)] $p(S, T)=0$, and
\item[(iii)] the Taylor spectrum of $(S, T)$ is 
       contained in the closure of $\V$ 
       in $\overline{\mathbb D^2}$. 
\end{itemize}
\end{theorem}

Theorem \ref{thm:isometries}, with additional detail, is proved in
Section \ref{sec:shifts}.  See \cite{AKM} for
more on pairs $(S,T)$ satisfying (i) and (ii) above for some
polynomial $p$.

\subsubsection{Dilating commuting isometries with spectrum in $\overline{\V}$}

\begin{theorem}
 \label{thm:dilation-with-relations}
  Let $K$ be an admissible kernel.  
  Then the isometries $(S, T)$ of the previous theorem admit a 
 commuting unitary extension $(X, Y)$ such that $p(X, Y)=0$ 
 and the joint spectral measure for $(X, Y)$ lies in $\partial \V$.   

   In fact, 
 \begin{equation*}
  \begin{split}
    X=& \begin{pmatrix} S & \Sigma\\ 0 & \tS^*\end{pmatrix} \\
    Y=& \begin{pmatrix} T & \Gamma\\ 0 & \tT^* \end{pmatrix},
  \end{split}  
 \end{equation*}
   for a canonical pair of operators $\Sigma,\Gamma :H^2(\tK)\to H^2(K)$
  and where $S,T,\tilde{S},\tilde{T}$ are defined immediately after
  Remark \ref{rem:tK}.
\end{theorem}
  This theorem is proved in Section \ref{subsec:dilation-with-relations}.

%Write $\|f\|_\infty$ for the supremum norm of a bounded analytic
%function $f$ on $\V$:
%\begin{equation*}
%  \|f\|_\infty :=\sup_{(z,w)\in \V}|f(z,w)|
%\end{equation*}
If $f\in H^\infty_{\mathcal K}(\V)$, it is always the case that $\|f\|_{\infty}\leq
\|f\|_{\V}$.  If $q$ is a polynomial, then the operator $M_q$ on
$H^2(K)$ is equal to $q(S,T)$.  The following corollary is then
immediate from (\ref{eqn:V_norm_is_mult_norm}) and
Theorem~\ref{thm:dilation-with-relations}:
\begin{corollary}\label{cor:polys_are_multipliers}
Every polynomial $q(z,w)$ belongs to $H^\infty_{\mathcal K}(\V)$, and $\|q\|_{\V}=\|q\|_\infty$. 
\end{corollary}
To extend this result from polynomials to all of $H^\infty (\V)$, we have the following approximation theorem and its corollary.
\begin{theorem}\label{thm:poly-approx}  For each  $f\in H^\infty(\V)$, there exists a sequence of polynomials $p_n$ such that $p_n\to f$ uniformly on compact subsets of $\V$ and $\|p_n\|_\infty\leq \|f\|_\infty$ for all $n$.  
\end{theorem}
\begin{corollary}\label{cor:isometry-of-algebras}
  $H^\infty_{\cK}(\V)=H^\infty(\V)$ isometrically.
\end{corollary}
Theorem~\ref{thm:pick-final} is then immediate from Theorem~\ref{thm:interpolate} and Corollary~\ref{cor:isometry-of-algebras}.

%We have introduced the algebra $H^\infty_{\mathcal K}(\V)$; of course
%the ``natural'' algebra one would like to interpolate in is
%$H^\infty(\V)$, the algebra of all bounded analytic functions on $\V$.
%(By definition, a function $f$ on $\V$ is analytic at a point $v\in\V$
%if there is a neighborhood $U\subset \mathbb D^2$ of $v$ and an analytic
%function $g$ on $U$ such that the restriction of $g$ to $U\cap \V$ is
%the same as the restriction of $f$ to $U\cap\V$.)
%In Section~\ref{sec:analytic-functions} we investigate the relationship
%between $H^\infty_{\mathcal K}$ and $H^\infty(\V)$.  Let
%$P^\infty(\V)$ denote the set of all functions on $\V$ that can be
%approximated pointwise on $\V$ by a uniformly bounded sequence of
%polynomials.   Then:

%\begin{theorem}\label{thm:inclusions}
%Let $\mathcal V$ be a distinguished variety.  Then
%\begin{equation}\label{eqn:inclusions}
%P^\infty(\mathcal V)\subseteq H^\infty_{\mathcal K}(\mathcal V)\subseteq H^\infty(\mathcal V),
%\end{equation}
%and the second inclusion is contractive.
%\end{theorem}

%In particular, every function $f\in H^\infty_{\mathcal K}(\V)$ is
%bounded and analytic on $\V$, and if $f$ is pointwise approximable by
%polynomials $p_n$ with $\|p_n\|_\infty \leq C$, then $f\in
%H^\infty_{\mathcal K}(\V)$ with $\|f\|_{\V}\leq C$.  In all of the
%examples we have been able to compute, $H^\infty_{\mathcal
%  K}(\V)=H^\infty(\V)$ isometrically.  In fact we know of no examples
%where $P^\infty(\V)\neq H^\infty(\V)$.

\subsection{Readers Guide}
 \label{subsec:guide}
The remainder of the paper is organized as follows.
Section~\ref{sec:examples} considers a number of examples of the Pick
interpolation theorem.  The result for $H^\infty_{\mathcal K}$, Theorem \ref{thm:interpolate}, is
proved in Section \ref{sec:kernel-structures} by verifying the
collection of admissible kernels satisfies the conditions of the
abstract interpolation theorem from \cite{JKM-kernels}.  Section
\ref{sec:ad-pairs} develops facts about admissible pairs, admissible
kernels and determinantal representations needed for the sequel.
Section \ref{sec:shifts} treats the pairs of operators that play the
role of bundle shifts on $\V$. It contains proofs of
Theorems~\ref{thm:isometries} and \ref{thm:dilation-with-relations}.
Theorem~\ref{thm:poly-approx} is proved in
Section~\ref{sec:approx}. The proof here is function-theoretic and
independent of the other sections.  Finally, in Section
\ref{sec:analytic-functions} we prove
Corollary~\ref{cor:isometry-of-algebras}, and in particular show that
elements of $H^\infty_{\mathcal K}(\V)$ are actually holomorphic on
$\V$.

\section{Examples} 
 \label{sec:examples}
  It is instructive to consider a couple of examples
  related to existing Pick interpolation theorems.

\subsection{The Neil Parabola}
 \label{sec:Neil}
  The Neil parabola $\mathcal N$ is the distinguished variety determined
  by the polynomial $p(z,w)=z^3-w^2$.  Note the singularity
  at the origin.  It is easily checked that the pair
 \begin{equation*}
  \begin{split}
    Q(z,w)=& \begin{pmatrix} 1 & w \end{pmatrix} \\
    P(z,w)=& \begin{pmatrix} 1& z & z^2 \end{pmatrix}
  \end{split}
 \end{equation*}
  is an admissible pair with corresponding reproducing kernel
 \[
   \frac{1 + w\oeta}{1-z\ozeta}
       = K((z,w),(\zeta,\eta)) = \frac{1+z\ozeta +z^2\ozeta^2}{1-w\oeta}.
 \]
   Again, we emphasize that the equalities hold on $\V\times \V$.  Similarly, the pair
 \[
  \begin{split}
   Q(z,w)=& \begin{pmatrix} z & w\end{pmatrix}\\
   P(z,w)=& \begin{pmatrix} w & z & z^2 \end{pmatrix}
  \end{split}
 \]
  is admissible. The corresponding kernel vanishes at $((0,0),(0,0))$.

The next proposition identifies the algebra $H^\infty_{\mathcal K}(\mathcal N)$ with  $\mathcal A:=\{f\in
  H^\infty(\mathbb D): f^\prime(0)=0\}$.  

 \begin{proposition}
  \label{prop:Neil}
  The parametrization $\Psi:\mathbb D \to \mathcal N$ of the  Neil parabola
  given by $\Psi(t)=(t^2,t^3)=(z(t),w(t))$ induces an isometric isomorphism
  $\Psi^*: H^\infty_{\mathcal K}(\N) \to \mathcal A$ defined by $\Psi^*f(t)=f(\Psi(t))$.
 \end{proposition}

 \begin{proof}
   Since every $f\in H^\infty_{\mathcal K}(\N)$ is bounded,
   the mapping $\Psi^*$ does map into $H^\infty(\mathbb D)$.
   On the other hand, $(f\circ \Psi)^\prime(0)=Df(\Psi(0))\cdot D\Psi(0)=0$.
   Thus $\Psi^*$ maps into $\mathcal A$. Moreover,
   if $f\in H^\infty_{\mathcal K}(\N)$, then 
 \[
   \|f\|_{\N} \ge \|f\|_\infty \geq \|\Psi^*f\|
 \]
   and thus $\Psi^*$ is contractive.  It remains to show that $\Psi^*$
   is isometric and onto.  We first prove this for polynomials; the
   general case will follow by approximation.

So, let $p\in \mathcal A$ be a polynomial.  Then $p(t)=p_0 +
\sum_{j=2}^N p_j t^j$.  Each $j\ge 2$ can be written as
$j=2\alpha+3\beta$ for non-negative integers $\alpha,\beta$. (Of
course this representation is not unique; we fix one such for each
$j$.)  Let $q(z,w)=p_0 +\sum_{\alpha,\beta} p_{2\alpha+3\beta}z^\alpha
w^\beta$.  Thus $(\Psi^*q)(t)=p(t)$.  If we write $\|p\|$ for the
supremum of $|p|$ over the unit disk, it follows that
$\|q\|_\infty=\|p\|$.  By Corollary~\ref{cor:polys_are_multipliers},
$\|q\|_{\mathcal N} =\|p\|$.

Now suppose that $g\in \mathcal A$ is arbitrary.  There exists a
sequence of polynomials $p_n\in\mathcal A$ such that $\|p_n\|\le
\|g\|$ and $p_n$ converges pointwise to $g$ (e.g., one can take $p_n$
to be the $n^{th}$ Ces\'aro mean of the Fourier series of $g$).  For
each $n$ there is a polynomial $q_n$ such that $\Psi^*(q_n)=p_n$ and
$\|q_n\|_{\N} =\|p_n\|$.  Then $(q_n)$ 
is Cauchy (as $(p_n)$ is) and hence converges pointwise
 and in norm  to a function
$f\in H^\infty_{\mathcal K}$ (see Proposition \ref{props-of-H-infty})
  satisfying $\Psi^*f=g$.%, and by Corollary~\ref{cor:isometry-of-algebras}, this
%$f$ belongs to $H^\infty_{\mathcal K}(\N)$ and $\|f\|_{\N}\leq \|g\|$.  Thus,
By construction, $\|f\|_{\N}=\|g\|$. Thus
$\Psi^*$ is isometric and onto, and the proof is complete.

Note that the argument in this proof directly establishes
 Corollary \ref{cor:isometry-of-algebras} 
  for the Neil parabola. 
\end{proof}

Interpolation on the Neil parabola is thus equivalent to the
constrained Pick interpolation in the algebra $\mathcal A$, which was
considered by \cite{DPRS} and also by \cite{BBtH, DP}.  In \cite{DPRS}
it is shown that for the Pick interpolation problem in $\mathcal A$,
it suffices to consider the family of kernels
\begin{equation}\label{E:DPRS_family}
k_{a,b}(s,t) =(a+bs)\overline{(a+bt)} + \frac{s^2\overline{t^2}}{1-s\overline{t}}
\end{equation}
over all complex numbers $a,b$ with $|a|^2+|b|^2=1$.  On the other
hand, Proposition~\ref{prop:Neil} and Theorem~\ref{thm:pick-final}
say that it suffices to consider the family of kernels obtained by
pulling back admissible kernels on $\mathcal N$ under the map $\Psi$,
that is, all kernels on $\mathbb D\times \mathbb D$ of the form
\begin{equation}\label{E:neil_pullback_kernels}
K(s,t) =\frac{Q(s^2, s^3)Q(t^2, t^3)^*}{1-s^2\overline{t^2}} =\frac{P(s^2, s^3)P(t^2, t^3)^*}{1-s^3\overline{t^3}}
\end{equation} 
where $P, Q$ are an admissible pair.  It is not hard to see that the
latter family contains the former (up to conjugacy).  Indeed, given
$a,b$, define
\begin{align*}
Q(z,w) &= (az+bw \quad \overline{b}z^2 -\overline{a} zw ),\\
P(z,w)&= (az+bw \quad \overline{b}zw-\overline{a}w^2 \quad z^2).
\end{align*}
This is an admissible pair, and if $K$ is then defined by
\eqref{E:neil_pullback_kernels}, then
\begin{equation*}
K(s,t)=s^2k_{a,b}(s,t)\overline{t}^2.
\end{equation*}

 In forthcoming work we show that fairly generally for distinguished varieties
  that it is necessary to consider only the scalar kernels and that moreover
 in the case of the Neil parabola we then obtain the result of \cite{DPRS}.
  
\subsection{The annulus}
 \label{sec:annulus}

%Let $R$ be a bordered Riemann surface.  Suppose we are given an
%algebraic pair of inner functions $\psi_1,\psi_2$ on $R$; that is,
%there exists a polynomial $p$ such that $p(\psi_1,\psi_2)=0$.  Let
%$\mathcal A$ denote the weak-* closed unital subalgebra of
%$H^\infty(R)$ generated by $\psi_1,\psi_2$.  Then since the $\psi_j$
%are inner, $\V:=\mathcal Z_p\cap\mathbb D^2$ is necessarily a
%distinguished variety, paramaterized by
%$\Psi(t)=(\psi_1(t),\psi_2(t))$.  (In fact, every distinguished
%variety arises this way, by considering the desingularization of the
%algebraic curve $\mathcal Z_p$, but we will not need this fact.)  The
%proof of Proposition~\ref{prop:Neil} adapts easily to prove that
%$\Psi^*:H^\infty(\V) \to \mathcal A$ is an isometric isomorphism.

Fix $0<r<1$ and consider the annulus $\mathbb A :=\{z\in\mathbb C:
r<|z|<1\}$.  Let $A(\mathbb A)$ denote the algebra of functions
analytic in $\mathbb A$ and continuous in $\overline{\mathbb A}$, and
$H^\infty(\mathbb A)$ the algebra of all bounded analytic functions in
$\mathbb A$, both equipped with the uniform norm.  By a theorem of
Pavlov and Fedorov \cite{PF}, there exists an algebraic pair of inner
functions $\theta_0,\theta_1$ in $A(\mathbb A)$ such that the
polynomials in $\theta_0,\theta_1$ are dense in $A(\mathbb A)$.  Here
``algebraic'' means that there is a polynomial $p$ such that
$p(\theta_0,\theta_1)=0$.  Since the $\theta_j$ are inner, this $p$
must define a distinguished variety, which we denote $\V$.  Now, if
$g\in H^\infty(\mathbb A)$, it is not hard to prove that there exists
a sequence of functions $f_n\in A(\mathbb A)$ such that $f_n\to g$
pointwise and $\|f_n\|\leq \|g\|$ for all $n$.  Evidently these $f_n$
may be taken to be polynomials in $\theta_0,\theta_1$.  Imitating the
proof of Proposition~\ref{prop:Neil} (with $\theta_0,\theta_1$ in
place of the inner functions $t^2, t^3$) gives
 \begin{proposition}\label{prop:annulus}
   The spaces $H^\infty_{\mathcal K}(\V)$ and $H^\infty(\mathbb A)$
   are isometrically isomorphic.
 \end{proposition}
Of course, this can also be obtained as a special case of
Corollary~\ref{cor:isometry-of-algebras}. The prototype of Pick
interpolation theorems involving a family of kernels is the
interpolation Theorem of Abrahamse on multiply connected domains
\cite{Ab}. (See also \cite{Sarasonannulus}.)  Thus while Theorem~\ref{thm:pick-final} (actually its refinement
 to just scalar admissible kernels) gives an interpolation condition for the
annulus, we do not know if it is the same as that of
Abrahamse.  More generally, the results of Raghupathi \cite{R2} and Davidson-Hamilton \cite{DH} can be applied to obtain an interpolation theorem on distinguished varieties, by first lifting to the desingularizing Riemann surface and then uniformizing this surface as the quotient of the disk by a Fuchsian group. The interpolation theorem of \cite{R2} does indeed reduce to that of Abrahamse in the case of the annulus.  So, more generally, the question is open whether or not Theorem~\ref{thm:pick-final} gives the same conditions as those of \cite{R2}.

\section{Kernel structures}
 \label{sec:kernel-structures}
  Our proof of Theorem \ref{thm:interpolate}
  relies on an application of the main result
  from \cite{JKM-kernels} which, for the reader's
  convenience, we outline in this section.
  This approach to Pick interpolation complements that in \cite{Ag}.

  Let $M_n$ denote the set of $n\times n$ matrices with
  entries from  $\mathbb C$. An  
  $M_n$-valued kernel on a set $X$ is a positive
  semi-definite function $k:X\times X \to M_n$.
  Of course, our admissible kernels on $\V$
  are examples.  In what follows we use $z^*$
  to denote the complex conjugate of a complex
 number $z$ (anticipating that the results 
 are valid for matrix-valued functions). 

\begin{definition}
 \label{def:kernel-family}
   Fix a set $X$ and a sequence $\mathcal K=(\mathcal K_n)$
   where each $\mathcal K_n$ is a set of $M_n$-valued kernels
   on $X$.  

   The collection $\mathcal K$ is an {\em Agler interpolation
   family of kernels} provided:
 \begin{itemize}
      \item[(i)] if $k_1 \in \mathcal K_{n_1}$ and $k_2\in\mathcal K_{n_2}$, 
       then  $k_1\oplus k_2 \in \mathcal K_{n_1+n_2}$;
   \item[(ii)] if $k\in\mathcal K_n$, $z\in X$, 
         $\gamma\in \mathbb C^n$, and $\gamma^* k(z,z)\gamma\neq 0$,  then
     there exists an $N$, a kernel  $\kappa \in\mathcal K_N$,
     and a function $G:X\to M_{n,N}$ such that     
    \begin{equation*}
      k^\prime(x,y)  :=  k(x,y)
          - \frac{k(x,z)\gamma \gamma^* k(z,y)}{\gamma^*k(z,z)\gamma} 
          = G(x) \kappa(x,y)G(y)^*;
    \end{equation*}
   \item[(iii)] for each finite $F\subset X$
    and for each  $f:F\to \mathbb C$,  there is a $\rho>0$ 
       such that, for each
       $k\in\mathcal K$,
    \begin{equation*}
      F\times F \ni (x,y) \mapsto (\rho^2 -f(x)f(y)^*)k(x,y) 
    \end{equation*}
      is a positive semi-definite kernel on $F$; and
    \item[(iv)] for each $x\in X$ there is a $k\in\mathcal K$ 
      such that $k(x,x)$
      is nonzero (and positive semi-definite).
 \end{itemize}
\end{definition}

\begin{theorem}{\cite[Theorem 1.3]{JKM-kernels}}
 \label{thm:main}
  Suppose $\mathcal K$ is an Agler interpolation family of kernels on $X$.
  Further suppose $Y\subset X$ is finite
  and $g:Y\to \mathbb C$ and $\rho\ge 0$.
  If for each $k\in\mathcal K$ the kernel
 \begin{equation}
  \label{eq:thm-main-Y}
    Y\times Y \ni (x,y) \mapsto (\rho^2 -g(x)g(y)^*)k(x,y) 
 \end{equation} 
  is positive semi-definite, then there exists $f:X\to \mathbb C$ 
  such that $f|_Y = g$ and for each $k\in\mathcal K$ the kernel
 \begin{equation}
  \label{eq:thm-main-X}
      X\times X \ni (x,y) \mapsto (\rho^2 -f(x)f(y)^*)k(x,y) 
 \end{equation}
      is positive semi-definite.
\end{theorem}

  That the collection $\mathcal A$ of admissible kernels on $\V$ is
  an Agler-interpolation family is proved in the following 
 subsections. Theorem \ref{thm:interpolate} then follows.

  The direct sum of admissible kernels is evidently admissible and a
  result of \cite{Kn2} says that there is an admissible kernel $K$ on
  $\V$ such that $K((z,w),(z,w))$ does not vanish on $\V$. (See
  Theorem 11.3 of \cite{Kn2}.) Hence $\mathcal A$ satisfies conditions
  (i) and (iv).
  
%% Gave precise reference to non-vanishing result. --Greg

\subsection{Compression Stability}

  That $\mathcal A$ satisfies condition (ii) 
  of Definition \ref{def:kernel-family} is proved 
  in this subsection.  

We begin with the observation that condition (ii) in the definition of
an admissible pair (Definition~\ref{def:admissablePQ}) implies a
stronger version of itself; this will be needed in the proofs of both
Theorem~\ref{thm:interpolate} and Lemma~\ref{lem:Qinvertible} below.

\begin{lemma}\label{lem:full_rank_ae}  If $Q(z,w)$ is an $M_{\alpha,m\alpha}$-valued polynomial and $Q$ has
full rank at some point of each irreducible component of $\mathcal
Z_p$, then $Q$ has full rank at all but finitely many points of
$\mathcal Z_p$.
\end{lemma} 
\begin{proof}  It suffices to assume that $\Zp$ is irreducible.
  Choose $(z_0,w_0)\in \mathcal Z_p$ such that $Q(z_0,w_0)$ has full
  rank.  In particular, there are $\alpha$ columns of $Q$ which form a
  linearly independent set when evaluated at $(z_0,w_0)$. Choose such
  a set of columns and let $R(z,w)$ denote the resulting $\alpha\times
  \alpha$ matrix-valued polynomial. The polynomial $q=\det(R)$ does
  not vanish at the point $(z_0,w_0)$, so the variety $U=\mathcal
  Z_p\cap \mathcal Z_q$ is a proper sub-variety of $\mathcal Z_p$.  By
  Bezout's theorem, $U$ is a finite set, and by construction $Q$ has
  full rank off $U$.
\end{proof}

Now, we may begin the proof.  Fix an admissible kernel $K$
corresponding to the rank $\alpha$ admissible pair $(P,Q)$, a point
$u=(x,y)\in \V,$ a vector $\gamma\in\mathbb C^\alpha$, and assuming
$K(u,u)\gamma\ne 0$, let
\begin{equation*}
 \begin{split}
  K^\prime((z,w),(\zeta,\eta)) = & \\
     \|K(u,u)^{\frac{1}{2}}\gamma\|^2 & K((z,w),(\zeta,\eta)) 
                  -K((z,w),u)\gamma\gamma^*K(u,(\zeta,\eta)).
 \end{split}
\end{equation*}

 Write
\begin{equation*}
 \delta =Q(x,y)^*\gamma \in\mathbb C^{m\alpha}
\end{equation*}
 and note we are assuming $\delta\ne 0$.
 From the definition of $K^\prime$ we have
\begin{equation*}
 K^\prime((z,w),(\zeta,\eta)) = 
 \left(\frac{Q(z,w)Q(\zeta,\eta)^*}{1-z\overline{\zeta}} 
   \right)\left( \frac{\|\delta\|^2}{1-|x|^2} \right) 
 -\frac{Q(z,w)\delta\delta^* Q(\zeta, \eta)}
  {(1-z\overline{x})(1-x\overline{\zeta})}.
\end{equation*}
 Let $\varphi_x$ denote the Mobi\"us map
\begin{equation*}
\varphi_x(z)= \frac{z-x}{1-z\overline{x}};
\end{equation*}
and recall the identity
\[
\frac{(1-z\bar{\zeta})(1-|x|^2)}{(1-z\bar{x})(1-x\bar{\zeta})} =
1-\varphi_x(z)\overline{\varphi_x(\zeta)}.
\]
Then we may rewrite $K^\prime$ as
\begin{equation}\label{eqn:kfactor1}
K^\prime((z,w);(\zeta,\eta))=\frac{Q(z,w)\left( \|\delta\|^2 -\delta\delta^* +\varphi_x(z)\overline{\varphi_x(\zeta)}\delta\delta^* \right)  Q(\zeta, \eta)^*}{(1-z\overline{\zeta})(1-|x|^2)}.
\end{equation}

  Now let $P_\delta$ denote the orthogonal projection of $\mathbb C^{m\alpha}$ 
  onto the one-dimensional subspace spanned by $\delta$.  Define
\begin{equation*}
  B(z):= P_\delta^\bot + \varphi_x(z) P_\delta .
\end{equation*}
  Observe that $B$ is analytic, contraction-valued in the unit disk, 
  and unitary on the unit circle.  We then have
\begin{equation}\label{eqn:kfactor2}
 \|\delta\|^2 -\delta\delta^* +\varphi_x(z)\overline{\varphi_x(\zeta)}\delta\delta^* =\|\delta\|^2 B(z)B(\zeta)^*.
\end{equation}
Finally, define
\begin{equation}\label{eqn:kfactor3}
 Q^\prime(z,w) := (1-w\overline{y})\frac{\|\delta\|}{\sqrt{1-|x|^2}} (1-z\overline{x})Q(z,w)B(z)
\end{equation}
Combining (\ref{eqn:kfactor1}), (\ref{eqn:kfactor2}), and (\ref{eqn:kfactor3}), we get
\begin{equation*}
 (1-w\overline{y})(1-y\overline{\eta})
  (1-z\overline{x})(1-x\overline{\zeta}) K^\prime 
 =  \frac{Q^\prime(z,w)Q^\prime(\zeta,\eta)^*}{1-z\overline{\zeta}}.
\end{equation*}
 An analogous construction produces a 
 $P^\prime$ so that 
\[
  (1-w\overline{y})(1-y\overline{\eta})
  (1-z\overline{x})(1-x\overline{\zeta}) K^\prime 
  =  \frac{P^\prime(z,w)P^\prime(\zeta,\eta)^*}{1-w\overline{\eta}}.
\
\]
  From the construction of $Q^\prime$ it is of rank at most $m\alpha$
 and is a polynomial in $(z,w);$ and similarly for $P^\prime$. Also,
 the rank of $Q^\prime$ is the same as the rank of $Q$, except at the
 point $(x,y)$. Hence by Lemma~\ref{lem:full_rank_ae}, $Q^\prime$ has full rank at some point on each
 irreducible subvariety of $\mathcal Z_p$ (indeed, at all but finitely many points).  Thus
\[
  \kappa((z,w),(\zeta,\eta))
    =\frac{Q^\prime(z,w)Q^\prime(\zeta,\eta)^*}{1-z\overline{\zeta}}
   = \frac{P^\prime(z,w)P^\prime(\zeta,\eta)^*}{1-w\overline{\eta}}
\] 
 is an admissible kernel and
\[
  K^\prime =\frac{1}{(1-z\overline{x})(1-w\overline{y})}
     \kappa \frac{1}{(1-x\overline{\zeta})(1-y\overline{\eta})}.
\]

\subsection{Existence of interpolants}
 Finally, we verify condition (iii) of Definition \ref{def:kernel-family}.  
 Fix a finite set 
\begin{equation*}
  X=\{(z_1, w_1), \dots (z_N, w_N)\}\subset \V
\end{equation*} 
  and let $f:X\to \mathbb C$ be given.  Since polynomials separate
  points of $\mathbb C^2$, for each $j=1, \dots N$, we can choose a polynomial
  $p_j(z,w)$ such that $p_j(z_k,w_k)=\delta_{jk}$ for each $k=1, \dots N$.
%It is evident that polynomials %separate points on $\V$, so there
  %exist polynomials $q_1, \dots q_N$ such that $q_j(z_i,
  %w_i)=\delta_{ij}$ for all $i,j=1,\dots N$.
Now define
\begin{equation*}
  q(z,w)=\sum_{j=1}^N f(z_j, w_j)p_j(z,w)
\end{equation*}
Then $q\vert_X=f$.  Fix an admissible kernel $K$ and let $S=M_z,
T=M_w$.  As noted in the remark following Definition~\ref{def:Ks}, $S$
and $T$ are contractions, so by applying Ando's inequality to $q$, we
find that $\|q(S,T)\|=\|M_q\|_{B(H^2(K))}$ is bounded by the supremum
of $|q|$ over $\mathbb D^2$, and in particular is bounded independently
of $K$.  It then follows from equation (\ref{eqn:V_norm_is_mult_norm})
that $q\in H^\infty_{\mathcal K}(\V)$, and Definition \ref{def:kernel-family}(iii)
holds with $\rho=\sup_{(z,w)\in\mathbb D^2}|q(z,w)|$.

\begin{remark} \rm
In the verification of Definition \ref{def:kernel-family}(iii), the
bound on $\|q\|_{\V}$ coming from the above argument is quite crude;
if we appeal instead to Corollary~\ref{cor:polys_are_multipliers} we
obtain the sharp value $\|q\|_{\V}=\|q\|_\infty$.  We have
arranged the proof this way only to make the proof of the
interpolation theorem for $H^\infty_{\mathcal K}$ independent of the later dilation results.
\end{remark}

%  Then $q\vert_X=f$ and $q$ is a bounded multiplier of $H(K)$ 
%  (with norm $\|q\|_\infty$) for all admissible $K$ by 
%  Corollary~\ref{cor:polys_are_multipliers}.  Thus Definition 
%  \ref{def:kernel-family}(iii) holds with $\rho=\|q\|_{\V}$.  

\section{Admissible Pairs}
 \label{sec:ad-pairs}
 Recall that the variety $\Zp$ is the zero set of a square free
 polynomial $p(z,w)$, of bidegree $(n,m)$, as in the introduction
 and $\V=\Zp\cap \mathbb D^2.$ 

 \begin{lemma}
  \label{lem:Qinvertible}
  If $P,Q$ is an $\alpha$-admissible pair, then for all but finitely
  many $\lambda\in\mathbb D$ there exist distinct points $\mu_1,\dots,
  \mu_m\in\mathbb D\setminus\{0\}$ such that $(\lambda,\mu_j)\in \V$
  and the $m\alpha \times m\alpha$ matrices
\begin{equation*}
  \begin{pmatrix} Q(\lambda,\mu_1)^* & \dots & Q(\lambda,\mu_m)^*
   \end{pmatrix}, \ \ \ \
   \begin{pmatrix} Q(\frac{1}{\lambdastar},\frac{1}{\mustar_1})^* & \dots & 
          Q(\frac{1}{\lambdastar},\frac{1}{\mustar_m})^*
   \end{pmatrix}
\end{equation*}
  are invertible.
 \end{lemma}

 The proof of the lemma in turn uses a very modest generalization of
 the construction found in \cite{Kn1} with essentially
 identical proof involving the lurking isometry. See Lemma
 \ref{lem:PhiQP} below.  We record some preliminary observations.

\begin{lemma}
  \label{lem:not-rectangular}
  For each $z\in\mathbb C$ the set $\{w: (z,w)\in \Zp\}$ has
  cardinality at most $m$. Conversely, for all but at most finitely
  many $z$, this set has $m$ elements.
\end{lemma}

\begin{proof}
For fixed $z$, the polynomial $q(w) = p(z,w)$ has degree
less than or equal to $m$ in $w$.
Thus, $q(w)$ has at most $m$ zeros or is identically zero.
However, $q$ can't be identically zero since $\Zp \subset
\mathbb{D}^2\cup \mathbb{T}^2 \cup \mathbb{E}^2$.

Conversely, let us prove there are only finitely many $z$ at which
$\{w:(z,w) \in \Zp\} \ne m$.  First, note that there are only finitely
many $z$ at which $q(w) = p(z,w)$ has degree strictly less than $m$
(because the leading coefficient of $q$ is a polynomial in $z$).
Next, using the fact that $p$ is square free, it is not hard to show
that $\partial p/\partial w$ and $p$ have no common factors.
%% Here is a proof of this which might be overkill to include
\iffalse
Suppose $p$ and $p_w = \partial p/\partial w$ have a common
irreducible factor $f$.
Then, write $p = f G$, $p_w = f H$ for some polynomials $G,H$.  Taking
a derivative yields
\[
p_w = f_w G + f G_w = fH
\]
and this implies $f(H-G_w) = f_w G$.  (This is nonzero since otherwise
$f$ would be a linear polynomial in $z$ alone, contradicting the fact
that $p$ is a distinguished variety.)  Since $f$ is irreducible, this
implies $f$ divides $G$.  This contradicts the fact that $p$ is square
free.  \fi Therefore, $p$ and $\partial p/\partial w$ have finitely
many common zeros.  Thus, as long as we avoid the finitely many $z$ at
which $q(w) = p(z,w)$ has degree less than $m$ and the finitely many
$z$ corresponding to (the first coordinate of) common roots of $p$ and
$\partial p/\partial w$, $q$ will be a polynomial of degree $m$ with
no multiple roots.  This proves the claim.
\end{proof}

 Let $U$ be a unitary matrix of size $(m+n)\alpha$ written in block
 $2\times 2$ form as
\[
  U = \begin{pmatrix} A & B \\ C & D \end{pmatrix}
\]
 with respect to the orthogonal sum 
 $\mathbb C^{m\alpha}\oplus \mathbb C^{n\alpha}$. To
 $U$ we associate the linear fractional, or transfer, function
\[
 \Phi(z) = A^* + C^*(I-zD^*)^{-1}zB^*.
\]
 Very standard calculations show that $\Phi$ is a rational matrix
 function with poles outside $\mathbb D$; is contractive-valued in
 $\mathbb D$; and unitary-valued (except for possibly finitely many
 points) on the boundary of $\mathbb D$.  Indeed, by Cramer's rule the
 entries of $(I-zD^*)^{-1}$ are rational functions of $z$, and since
 $\|D\|\leq 1$ they are analytic in $\mathbb D$.  Moreover, a short
 calculation using the fact that $U$ is unitary shows that
 \begin{equation}
   I-\Phi(z)^*\Phi(z) =(1-|z|^2) B(I-\overline{z}D)^{-1}(I-zD^*)^{-1} B^*
 \end{equation}
which is a positive matrix when $|z|<1$ (showing $\Phi(z)$ is contractive in
$\mathbb D$) and $0$ when $|z|=1$ (showing that $\Phi$ is unitary on
the circle, except at the finitely many points where $(I-zD^*)$ may
fail to be invertible).

 \begin{lemma}
  \label{lem:PhiQP}
    If $(P,Q)$ is an $\alpha$-admissible pair, 
    then there exists an $m\alpha\times m\alpha$ 
    matrix-valued transfer function such that 
  \begin{equation*}
   \Phi(z)^*Q(z,w)^* = w^* Q(z,w)^*
  \end{equation*} 
   for all $(z,w)\in V$
  and for all except finitely many $(z,w)\in \Zp.$
  
   Moreover, $p$ divides $\det(\Phi(z)-wI)$.
 \end{lemma}

\begin{proof}  %%[Sketch for now]
    First, we rearrange the relation
\[
\frac{Q(z,w)Q(\zeta,\eta)^*}{1-z\bar{\zeta}} =
\frac{P(z,w)P(\zeta,\eta)^*}{1-w\bar{\eta}}
\]
for $(z,w), (\zeta,\eta) \in \V$ into the isometric form
   \begin{equation}
    \label{eq:lurk}
      QQ^* + z\zetastar PP^* = w\etastar QQ^* +PP^*.
   \end{equation}
     Let
   \begin{equation*}
    \label{eq:EEs}
     \begin{split}
      \mathcal E=&\text{span}\{\begin{pmatrix} Q(\zeta,\eta)^* \gamma \\
            \zetastar P(\zeta,\eta)^*\gamma \end{pmatrix}: (\zeta,\eta)\in \V, 
                \gamma\in \mathbb C^\alpha \} \\
     \mathcal F=&\text{span} \{\begin{pmatrix} \etastar  Q(\zeta,\eta)^* \gamma \\
               P(\zeta,\eta)^*\gamma\end{pmatrix}: (\zeta,\eta)\in \V, 
                \gamma\in \mathbb C^\alpha \}.
     \end{split}
   \end{equation*}
    %% Note that both $\mathcal E$ and $\mathcal F$ are
    %% subspaces of $\mathbb C^{(m+n)\alpha}$.
 Equation \eqref{eq:lurk} implies that the mapping
   \begin{equation}
    \label{eq:defV}
      U \begin{pmatrix} Q(\zeta,\eta)^* \gamma \\
            \zetastar P(\zeta,\eta)^*\gamma \end{pmatrix}
             = \begin{pmatrix} \etastar  Q(\zeta,\eta)^* \gamma \\
               P(\zeta,\eta)^*\gamma\end{pmatrix}
   \end{equation}
   determines a well defined isometry $U:\mathcal E\to\mathcal F$.
   Since we are in finite dimensions, it follows that $U$ extends to a
   unitary (which we still denote $U$) on $\mathbb
   C^{(m+n)\alpha}$. Write
  \begin{equation*}
    U=\begin{pmatrix} A & B \\ C & D \end{pmatrix}:
     \begin{matrix}\mathbb C^{m\alpha} \\ \oplus\\ \mathbb C^{n\alpha}\end{matrix}
     \to \begin{matrix}\mathbb C^{m\alpha} \\ \oplus\\ \mathbb C^{n\alpha} \end{matrix}
  \end{equation*}
   and define
  \begin{equation*}
   \Phi(\zeta)^*= A+\zetastar B(I-\zetastar D)^{-1}C.
  \end{equation*}

By definition of $U$,
\[
\begin{aligned}
A Q(\zeta,\eta)^*\gamma + B \zetastar P(\zeta,\eta)^*\gamma & =
\etastar Q(\zeta,\eta)^*\gamma \\
C Q(\zeta,\eta)^*\gamma + D \zetastar P(\zeta,\eta)^*\gamma &=
P(\zeta,\eta)^*\gamma
\end{aligned}
\]
which implies
\[
(I-\zetastar D)^{-1} C Q(\zeta,\eta)^* \gamma = P(\zeta,\eta)^* \gamma 
\]
and therefore
\[
A Q(\zeta,\eta)^*\gamma + \zetastar B (I-\zetastar D)^{-1}C
Q(\zeta,\eta)^*\gamma = \Phi(\zeta)^* Q(\zeta,\eta)^*\gamma = \etastar 
Q(\zeta,\eta)^* \gamma.
\]
for all $\gamma \in \mathbb{C}^{\alpha}$.  So, we indeed have
\[
\Phi(z)^* Q(z,w)^* = w^* Q(z,w)^*
\]
everywhere on $\Zp$, excluding the finitely many points $(z,w)$ where
$\Phi(z)$ may not be defined. 
\end{proof}

%Finally, $\Phi(z)$ is unitary on the circle and analytic (and
%contraction-valued) in the disk
%because of the equation
%\[
%U \begin{pmatrix} I \\ \zeta^* (I-\zeta^* D)^{-1}C \end{pmatrix}
%= \begin{pmatrix} \Phi(\zeta)^* \\ (I-\zeta^* D)^{-1} C \end{pmatrix}.
%\]
%Namely, using the fact that $U$ is unitary we have
%\[
% \begin{split}
%I &+ |\zeta|^2 C^*(I-\zeta D^*)^{-1}(I-\zeta^*D)^{-1}C \\
% &  =
%\Phi(\zeta)\Phi(\zeta)^* + C^*(I-\zeta D^*)^{-1}(I-\zeta^* D)^{-1}C
% \end{split}
%\]
%which can be rearranged into
%\[
%I-\Phi(\zeta)\Phi(\zeta)^* = 
%   (1-|\zeta|^2)  C^*(I-\zeta D^*)^{-1}(I-\zeta^* D)^{-1}C.
%\]
%The above says $\Phi$ is unitary on the circle.

Note in passing that we also have:
 \begin{equation*}
  \begin{split}
    \frac{1}{\zetastar}P^*=&[D+C(\etastar -A)^{-1}B]P^* \\
    \frac{1}{\etastar } Q(\zeta,\eta)^*=& [A^* + C^*(\zetastar -D^*)^{-1}B^*]Q^* \\
    \zetastar P^* = & [D^* +\etastar  (I-\etastar A^*)^{-1} C^*]P^*.    
  \end{split}
 \end{equation*}
%\end{proof}

\begin{proof} [Proof of Lemma \ref{lem:Qinvertible}]
Choose $\lambda\in \mathbb D$ such that 
\begin{enumerate}[(i)]
\item there are $m$ {\it distinct} points $\mu_1,\dots,\mu_m\in
  \mathbb D$ such that $(\lambda,\mu_j) \in \V$, and
 \item for each $j=1,\dots m$, the matrices $Q(\lambda, \mu_j)$ and
  $Q(\frac{1}{\lambdastar},\frac{1}{\mustar_j})$ have rank $\alpha$.
\end{enumerate}  
(This is possible by combining Lemmas~\ref{lem:not-rectangular} and
\ref{lem:full_rank_ae}.)  Let $\Phi$ denote the rational function from
Lemma \ref{lem:PhiQP}. For $\gamma\in\mathbb C^{\alpha}$,
 \begin{equation*}
   \Phi(\lambda)^*Q(\lambda,\mu_j)^*\gamma 
       = \mustar_j Q(\lambda,\mu_j)^*\gamma.
 \end{equation*}   
   Thus $Q(\lambda,\mu_j)^*\gamma$ is in the eigenspace of
   $\Phi(\lambda)^*$ corresponding to the eigenvalue $\mustar_j$, and
   for each $j$ this eigenspace has dimension $\alpha$. It follows
   that the matrix $(Q(\lambda,\mu_j)^*)_j$ has rank $m\alpha$;
   similarly for $(Q(\frac{1}{\lambdastar},\frac{1}{\mustar_j})^*)_j$.
\end{proof}

\section{Bundle Shifts}
 \label{sec:shifts}
Recall that $(S,T) = (M_z,M_w)$ on $H^2(K)$ and $(\tS, \tT) =
(M_{1/z}, M_{1/w})$ on $H^2(\tK)$. 
 See Definition \ref{def:Ks}.

\begin{theorem}
 \label{thm:isometries-more}
    Suppose $(P,Q)$ is a rank $\alpha$-admissible pair.
    The operators $S,\tS,T,\tT$ are all pure isometries.
    Moreover, the Taylor spectra of $(S,T)$ and $(\tS^*,\tT^*)$ are
    contained in $\Zp \cap \overline{\mathbb D^2}$.
\end{theorem}

  That $p(S, T)=0$ is immediate, since $p(S, T) =M_p$ and 
\begin{equation*}
 M_p^* K_{(z,w)} =\overline{p(z,w)}K_{(z,w)} =0
\end{equation*}
  for all $(z,w)\in \V$.  For the claim about the Taylor spectrum, 
  note that since both $S$ and $T$ are contractions,
  $\sigma_{Tay}(S,T)\subset \overline{\mathbb D^2}$. Further, 
  by Taylor's mapping theorem we have
\begin{equation*}
 \{0\}=\sigma(p(S, T)) =p(\sigma_{Tay}(S, T)) 
\end{equation*}
 so $\sigma_{Tay}(S, T) \subset \Zp \cap \overline{\mathbb D^2}$.  

  Let $\check{p}(z,w)=\overline{p(\oz,\ow)}$. Thus $\check{p}$ is obtained
  from $p$ by conjugating the coefficients of $p$.  A computation
  like that above gives 
\begin{equation*}
 \begin{split}
  \check{p}(\tS,\tT)^* \tK_{(\zeta,\eta)}
    = & \overline{\check{p}(\frac{1}{\zeta},\frac{1}{\eta})} \tK_{(\zeta,\eta)}\\
    =& p(\ts{\zeta},\ts{\eta})\tK_{(\zeta,\eta)} \\
    =& 0.
 \end{split}
\end{equation*}
   Thus $\check{p}(\tS,\tT)=0$. Equivalently, $p(\tS^*,\tT^*)=0$. 

  Our proof that $S$ is an isometry begins with a lemma.
  From the general theory of reproducing kernel Hilbert spaces
  (see \cite{AMbook} for instance),
  if $K$ is an $\alpha\times\alpha$ 
  matrix-valued  kernel, $g$ is a $\mathbb C^\alpha$-valued
  function on $\V,$ and
 \[
    K((z,w),(\zeta,\eta)) - g(z,w)g(\zeta,\eta)^*
 \]
  is positive semi-definite, then $g \in H^2(K)$, so in particular
 \[
     \langle  g,K_{(\zeta,\eta)}\gamma \rangle
          =\langle g(\zeta,\eta),\gamma\rangle_{\mathbb C^\alpha}.
 \]
   Given the admissible pair $(P,Q)$, let
 \[
    \mQ = Q(z,w)Q(\zeta,\eta)^*.
 \]
   It is immediate that $K-\mQ$ is a positive kernel. Hence,
   if $\gamma$ is a unit vector, then 
 \[
   K - Q(z,w)\gamma \gamma^*Q(\zeta,\eta)^* 
 \]
  is also a positive kernel and thus $Q(z,w)\gamma \in H^2(K)$.  The
  next lemma develops this observation further.
\begin{lemma}
 \label{lem:QQ-star} 
   Let $\mQ((z,w),(\zeta,\eta))=Q(z,w)Q(\zeta,\eta)^*$ and
   let $\cQ$ denote the span of $\{\mQ(\cdot,(\zeta,\eta))\gamma: 
    (\zeta,\eta)\in \V, \ \gamma\in\mathbb C^\alpha \}$.
    Then $\mQ$ is the reproducing kernel for $\cQ$ with
   respect to the inner product on $H^2(K)$; i.e.,
   $\mQ\in \cQ$ and 
   if $g\in \cQ$, then
 \begin{equation*}
    g(\zeta,\eta) = \langle g, \mQ(\cdot,(\zeta,\eta) \rangle.
 \end{equation*}
   Thus,
 \begin{equation*}
    P_{\cQ}K(\cdot,(\zeta,\eta))= \mQ(\cdot,(\zeta,\eta))
 \end{equation*}
   and
 \begin{equation*}
   K(\cdot,(\zeta,\eta))-P_{\cQ}K(\cdot,(\zeta,\eta))
    =\overline{\zeta} S K(\cdot,(\zeta,\eta)).
 \end{equation*}
\end{lemma}

  Let $\mathcal D=\{K_{(\zeta,\eta)} : (\zeta,\eta)\in \V, \, \zeta\ne 0\}.$
  The last identity in the Lemma implies 
  the set $S\mathcal D$
  is orthogonal to $\cQ$.  Since the span of $\mathcal D$ is dense
  in $H^2(K)$, it follows that 
  $SH^2(K)$ 
  is orthogonal to $\cQ$. 
  Using the Lemma and the fact that for $f\in H^2(K)$ and 
  $(\zeta,\eta)\in \V$,
\begin{equation*}
 \begin{split}
  \langle Sf, K(\zeta,\eta) \rangle
   = & \langle Sf, \overline{\zeta} S K(\cdot,(\zeta,\eta))\rangle
      + \langle Sf, P_{\cQ}K(\cdot,(\zeta,\eta)) \rangle \\
   = & \langle Sf, \overline{\zeta} S K(\cdot,(\zeta,\eta))\rangle.
 \end{split}
\end{equation*}
  Consequently,
\begin{equation*}
 \begin{split}
   \zeta \langle f, K(\cdot,(\zeta,\eta)) \rangle
   = & \zeta f(\zeta,\eta) = \langle Sf, 
          K(\cdot,(\zeta,\eta))\rangle \\
   = & \langle Sf, \overline{\zeta} SK(\cdot,(\zeta,\eta))\rangle.
 \end{split}
\end{equation*}
  It follows that $S$ is an isometry.

\begin{proof}[Proof of Lemma \ref{lem:QQ-star}]
The claim of the lemma may be understood as follows: the finite-dimensional space $\cQ$ can be made into a reproducing kernel Hilbert space in two ways. On the one hand, since $\cQ\subset H^2(K)$ we can simply restrict the norm from $H^2(K)$. On the other hand, we can define the kernel $\mQ$ as in the statement of the lemma and give $\cQ$ the norm coming from the resulting inner product.  The kernel in the first case is $P_{\cQ}K$, and the kernel in the second case is of course $\mQ$.  The lemma says that the two kernels are in fact equal, which is equivalent to saying that the associated Hilbert space norms are equal (since the norm determines inner product and the inner product determines the kernel).  To prove this, it suffices to prove that the identity map of $\cQ$ is contractive in both directions.  This may be proved by inspecting the kernels: the identity is contractive from the $H^2(\mQ)$ norm to the (restricted) $H^2(K)$ norm if and only if 
\begin{equation}
 \label{eq:KKQ}
     K \succeq \mQ ,
\end{equation}
while the map is contractive from the $H^2(K)$ norm to the $H^2(\mQ)$ norm if and only if for any $g \in \mathcal{Q}$
\begin{equation}
 \label{eq:gKQ}
 K \succeq gg^* \text{ implies } \mQ \succeq gg^*.
\end{equation}
(Recall from subsection \ref{subsubsec:H-infty} that
``$\succeq$'' represents an inequality in the sense of positive
semi-definite kernels; e.g. \eqref{eq:KKQ} says $K-\mQ$ is a positive
semi-definite kernel.)  These equivalences follow from the fact that for any reproducing kernel Hilbert space on a set $X$ with kernel $L$, a function $f:X\to \mathbb C$ lies in the unit ball of $H^2(L)$ if and only if $L\succeq ff^*$.

%  Reproducing kernel machinery says that the positive kernel $\mQ$ is
%  the reproducing kernel for $\cQ$ (i.e. 
%  $K_Q(\cdot,(\zeta,\eta)\in \cQ$ and satisfies the first condition
%  of the lemma) if both
%\begin{equation}
% \label{eq:KKQ}
%     K \succeq \mQ 
%\end{equation}
%and
%for any $g \in \mathcal{Q}$
%\begin{equation}
% \label{eq:gKQ}
% K \succeq gg^* \text{ implies } \mQ \succeq gg^*,
%\end{equation}
%where recall from subsection \ref{subsubsec:H-infty} that
%``$\succeq$'' represents an inequality in the sense of positive
%semi-definite kernels; e.g. \eqref{eq:KKQ} says $K-\mQ$ is a positive
%semi-definite kernel.  (In essence, these assertions imply that $\mQ$
%is bounded above and below by the reproducing kernel for $\cQ$.  The
%inequality of equation \eqref{eq:KKQ} says that the inclusion of
%$\cQ\subset H^2(\mQ)$ into $H^2(K)$ is contractive and the equation
%\eqref{eq:gKQ} says that the inclusion of $\cQ\subset H^2(K)$ into
%$H^2(\mQ)$ is also contractive.)

We may now begin the proof of the lemma.  First, it is clear that $K \succeq \mQ$ since
\[
 K-\mQ  = \zetastar SK.
\]

For the other direction, suppose 
\[
  g = \sum_{j=1}^{m} Q_j(z,w) \gamma_j = Q(z,w)\gamma 
\]
 is in $\mathcal{Q}$
 where   $\gamma = (\gamma_1 \dots \gamma_m)^t \in\mathbb{C}^{m\alpha}$ 
 and $K-gg^*\succeq 0$. 
 The inequality $K-gg^*\succeq 0$ is equivalent to positive semi-definiteness
 of the kernel
\begin{equation}
 \label{eq:QQ1}
  L(z,w)=Q(z,w)[\frac{1}{1-z\overline{\zeta}}I_{m\alpha} - \gamma\gamma^*] Q(\zeta,\eta)^*.
\end{equation}
    
  Fix $\lambda$ satisfying the conditions 
  of Lemma \ref{lem:Qinvertible}. In particular, 
  there are $m$ distinct points
  $\mu_1,\dots,\mu_m\in\mathbb D$ such that 
 $(\lambda,\mu_1), \dots, (\lambda,\mu_m) \in \V\cap \mathbb{D}^2.$ 
  The block matrix 
\[
 ( L((\lambda,\mu_j),(\lambda,\mu_k)) )_{j,k}
   = Z [\frac{1}{1-|\lambda|^2} - \gamma\gamma^*] Z^*
\]
  is positive definite, where 
\[
  Z^* = \begin{pmatrix} Q(\lambda,\mu_1)^* & \dots & Q(\lambda,\mu_m)^*
   \end{pmatrix}
\]
  (see  Lemma \ref{lem:Qinvertible}).  Since $Z$ is invertible,
  it follows that 
\[
  \left( \frac{1}{1-|\lambda|^2}I_{m\alpha} - \gamma\gamma^* \right) \ge 0.
\]
  Since, by Lemma \ref{lem:Qinvertible}, there is a sequence
  $\lambda_n\to 0$ for which this last inequality holds,
\[
  \left(I_{m\alpha} - \gamma\gamma^* \right) \ge 0.
\]

  It  now follows that 
\[
Q(z,w) (I_{m\alpha} - \gamma\gamma^*) Q(\zeta,\eta)^* \succeq 0.
\]
 This last inequality is equivalent to
\[
\mQ \succeq gg^*.
\]
  Thus, we have proved that $\mQ$ is the reproducing
  kernel for $\cQ$.    Since $P_{\cQ}K(\cdot,(\zeta,\eta))$ is also the
  reproducing kernel for $\cQ$, the second identity in the
  lemma follows.
%  Now suppose $g\in\cQ.$ Then, using what has already been proved,
% \begin{equation*}
%  \begin{split}
%   \langle g, \mQ(\cdot, (\zeta,\eta)\rangle 
%    = & g(\zeta,\eta)\\
%     = & \langle g, P_{\cQ} K(\cdot,(\zeta,\eta))\rangle. 
%  \end{split}
% \end{equation*}
%   Since also $\mQ(\cdot,(\zeta,\eta)\in\cQ$ the second
%   part of the Lemma follows.
The last statement follows from the identity
 \begin{equation*}
   K((z,w),(\zeta,\eta))-\mQ((z,w),(\zeta,\eta))
      = z\overline{\zeta} K((z,w),(\zeta,\eta)).
 \end{equation*}
\end{proof}

  To see that $S$ is pure, note that for any given $(\zeta,\eta)$
  and vector $\gamma$, 
 \begin{equation*}
    S^{*j} k(\cdot,(\zeta,\eta))\gamma = 
     \overline{\zeta}^j k(\cdot,(\zeta,\eta))\gamma.
 \end{equation*}
   Thus $S^{*j}f$ converges to $0$ for each $f$ 
   in the span of $\{K(\cdot,(\zeta,\eta))\gamma: (\zeta,\eta)\in \V, \ 
    \gamma \in \mathbb C^\alpha \}$. Since this set 
  is dense in $H^2(K)$ and since the sequence $S^{*j}$ is
  norm bounded, it follows that $S^{*j}$ converges to $0$
  in the SOT. Consequently, $S$ is a pure shift.

  Similar arguments show that $\tS$, $T$, and $\tT$ 
  are also pure isometries.  The following proposition
  identifies their defect spaces.

\begin{proposition}
 \label{prop:kerS-star}
   The kernel of $S^*$ is $\cQ$. Moreover,  
   if $\lambda,\mu_1,\dots,\mu_m$ satisfy the conditions
   of Lemma \ref{lem:Qinvertible}, then $\cQ$ is
  equal to the span of $\{ Q(\cdot)Q(\lambda,\mu_j)^*\gamma:
    \gamma\in\mathbb C^\alpha, \ 1\le j\le m\}.$
\end{proposition}

\begin{proof}
  As noted already, the subspace $\cQ$ is orthogonal to $SH^2(K)$.
  Hence $\cQ$ is a subspace of the kernel of $S^*$. 
  On the other hand, $\cQ+SH^2(K)=H^2(K)$, since
 \begin{equation*}
  K(\cdot,(\zeta,\eta))= \mQ(\cdot,(\zeta,\eta)) +
     \overline{\zeta}SK(\cdot,(\zeta,\eta))
 \end{equation*}
  and the first conclusion of the lemma follows.

  The dimension of $\cQ$ is
  at most $m\alpha$. On the other hand,
  under the hypothesis of the
  moreover part of the lemma the span of
  $\{ Q(\cdot)Q(\lambda,\mu_j)^*\gamma\}$ has
  dimension $m\alpha$.  
\end{proof}

\subsection{Proof of Theorem \ref{thm:dilation-with-relations}}
 \label{subsec:dilation-with-relations}
 In this subsection we give a proof of 
 Theorem \ref{thm:dilation-with-relations} based upon
 knowledge of the commutant of a pure shift. We
 sketch a second geometric proof 
 which identifies the extension in terms of the
 operators $\tS$ and $\tT$ canonically associated
 to $S$ and $T$ via the reflected kernel $\tK$. 

  The pure shift $S$ has multiplicity $m\alpha$ 
  and thus  can be modeled as multiplication by the coordinate function on a
vector valued Hardy space $H^2 \otimes \mathbb{C}^{m\alpha}.$
 Since $T$ commutes with $S$ and is itself a pure isometry
 of multiplicity $n\alpha$, 
 it is multiplication by a matrix valued rational inner function,
say $\Phi$, on $H^2\otimes \mathbb{C}^{m\alpha}$.  Therefore, the pair
$(S,T)$ can be thought of as the pair
\[
(M_z, M_\Phi) : H^2\otimes \mathbb{C}^{m\alpha} \to H^2\otimes
\mathbb{C}^{m\alpha}.
\]
We will necessarily have
\[
 p(zI,\Phi(z)) = 0
\]
since $p(M_z,M_w) = 0$.

This pair extends to a pair of unitary multiplication
operators on $L^2\otimes \mathbb{C}^{m\alpha}$.  The resulting pair of
unitaries will still satisfy the polynomial $p$ which defines the
distinguished variety in question.

 Next we sketch our geometric proof. 
 As in the proof of Lemma \ref{lem:QQ-star}, let 
 $\mQ(\cdot,(\zeta,\eta))=Q(\cdot)Q(\zeta,\eta)^*$
 and let $\cQ$ denote the span of
 $\{ \mQ(\cdot,(\zeta,\eta))\gamma:(\zeta,\eta)\in \V, \ \gamma
   \in \mathbb C^\alpha \}$.
 If $g\in \cQ$, then
\[
 \langle g,\mQ(\cdot,(\zeta,\eta))\gamma\rangle
   = \langle g(\zeta,\eta),\gamma \rangle.
\] 
  Since both sides are defined and analytic in $\mathcal Z_p$,
  it follows that the identity is valid  for $(\zeta,\eta)\in\tV$
  too. In particular, if also $\gamma^\prime \in\mathbb C^\alpha$, then
\[
  \langle \mQ(\cdot,(\zeta^\prime,\eta^\prime))\gamma^\prime,
    \mQ(\cdot,(\zeta,\eta))\gamma\rangle
     = \mQ((\zeta,\eta),(\zeta^\prime,\eta^\prime))\gamma^\prime,
        \gamma\rangle,
\]
  for $(\zeta,\eta), (\zeta^\prime,\eta^\prime)\in \tV$.

  By analogy with $\mQ$ and $\cQ$, let  
   $\tmQ((z,w),(\zeta,\eta))=Q(z,w)Q(\zeta,\eta)^*$
   for $(z,w),(\zeta,\eta)\in \tV$ and $\tcQ$ denote
  the span of $\{\tmQ(\cdot,(\zeta,\eta))\gamma: (\zeta,\eta)\in\tV, \
  \gamma\in\mathbb C^\alpha\}$. 
  
 Define
 $\Sigma,\Gamma :H^2(\tK)\to  H^2(K)$ by,
\begin{equation*}
 \begin{split}
  \Sigma \tK(\cdot,(\zeta,\eta)) = & \frac{1}{\overline{\zeta}}
            Q(\cdot)Q(\zeta,\eta)^* \\
  \Gamma \tK(\cdot,(\zeta,\eta)) = &
     \frac{1}{\overline{\eta}} P(\cdot)P(\zeta,\eta)^*.
 \end{split}
\end{equation*}

  Of course at this point $\Sigma$ and $\Gamma$ are only 
  densely defined.  The computations below show that
  $\Sigma^* \Sigma=P_{\tcQ}$, and similarly
  $\Gamma^*\Gamma,$ is a projection.

  Note that the functions on the left hand side
  are defined on $\tV$ and those on the right are
  defined on $\V$.

With these definitions of $\Sigma$ and $\Gamma$, the operators $X,Y$
on $H^2(K)\oplus H^2(\tK)$ from Theorem
\ref{thm:dilation-with-relations} are given by
\[
X = \begin{pmatrix} S & \Sigma \\ 0 & \tS^* \end{pmatrix} \qquad Y
= \begin{pmatrix} T & \Gamma \\ 0 & \tT^* \end{pmatrix} 
\]
and it is now our task to prove $X$ and $Y$ are commuting unitaries
satisfying $p(X,Y) = 0$.

 Compute, for $(\zeta,\eta),(\zeta^\prime,\eta^\prime)\in \tV$
  and $\gamma,\gamma^\prime\in\mathbb C^\alpha$, 
\[
 \begin{split}
  \langle \Sigma^* \Sigma \tK(\cdot,(\zeta^\prime,\eta^\prime))\gamma^\prime, 
    \tK(\cdot,(\zeta,\eta))\gamma \rangle
   = &\frac{1}{\zeta^\prime\overline{\zeta}}
        \langle \tmQ((\zeta,\eta),(\zeta^\prime,\eta^\prime))\gamma^\prime,
        \gamma \rangle \\
   =&  \langle P_{\tcQ} \tK(\cdot,((\zeta^\prime,\eta^\prime))\gamma^\prime,
            \tK(\cdot,(\zeta,\eta))\gamma \rangle.
 \end{split}
\]
  Thus, $\Sigma^* \Sigma=P_{\tcQ}.$ Hence, by Proposition
  \ref{prop:kerS-star}, $\Sigma^*\Sigma$ is the projection onto the
  kernel of $\tS^*$ and in particular
 \begin{equation}
   \label{eq:X1}
       I=\tS \tS^* +\Sigma^* \Sigma
 \end{equation}
  
   Since the range of $\Sigma$ is in $\cQ$, 
\begin{equation}
 \label{eq:X2}
   S^* \Sigma =0.
\end{equation}

  Using equations \eqref{eq:X1} and \eqref{eq:X2} and the
  fact that $S$ is an isometry,  it
  follows that the  $X$ in Theorem \ref{thm:dilation-with-relations}
  is an isometry; i.e., $X^*X=I$.

  Since $\Sigma\Sigma^*$ is a projection of rank $m\alpha$
  (same as the rank of $\Sigma^* \Sigma$)
  with range in the kernel of $S^*$, we conclude 
  $SS^*+\Sigma\Sigma^*=I$. Since also $\tS^*\tS=I$
  and $XX^* \le I$, it follows that $XX^*=I$.
  Hence $X$ is  unitary. A similar argument shows
  $Y$ is unitary.

  The commutation relation $XY=YX$ is equivalent to
\[
  S\Gamma - \Gamma \tS^* = T\Sigma -\Sigma \tT^*.
\]
  To see that this is indeed the case, compute,
\[
 \begin{split}
  \langle [S\Gamma -&\Gamma \tS^*]\tK(\cdot,(\zeta,\eta)\gamma,
             \tK(\cdot,(z,w))\delta  \rangle \\
          =& ((\frac{z}{\overline{\eta}}
   -\frac{1}{\overline{\zeta \eta}}) P(z,w)P(\zeta,\eta))^*\gamma,
              \delta \rangle \\
    = & \langle (\frac{z\overline{\zeta}-1}{\overline{\zeta\eta}} 
         P(z,w)P(\zeta,\eta)^* \gamma, \delta \rangle.
 \end{split}
\]
  Similarly,
\[
 \begin{split}
   \langle [T\Sigma -&\Sigma \tT^*]\tK(\cdot,(\zeta,\eta)\gamma,
             \tK(\cdot,(z,w))\delta  \rangle \\
   = & \langle (\frac{w\overline{\eta}-1}{\overline{\zeta\eta}})
        Q(z,w)Q(\zeta,\eta)^*\gamma,\delta \rangle.
 \end{split}
\]
  The commutation relation thus follows
  as a consequence of the fact that $(P,Q)$ is an admissible pair.
 
  For the statement about the spectrum, it is a property of the
  Taylor spectrum that, given the upper triangular structure
  of the pair $(X,Y)$ that 
\begin{equation*}
  \sigma_T(X,Y) \subset \sigma_T(S,T) \cup \sigma_T(\tS^*,\tT^*).
\end{equation*}
  The sets on the right hand  side both lie in $\mbox{closure}(\V)$.
  On the other hand, the projection property of the Taylor
  spectrum implies,
\begin{equation*}
 \sigma_T(X,Y) \subset \sigma(X) \times \sigma(Y) \subset \mathbb T\times \mathbb T.
\end{equation*}
 Putting the last two inclusions together it follows that
 $\sigma_T(X,Y) \subset \partial \V$.

\section{Polynomial approximation on $\V$}\label{sec:approx}

%%Setup and notation:
This section proves the fundamental and function-theoretic Theorem
\ref{thm:poly-approx}. It is largely independent from other
sections.
 
Suppose $p \in \mathbb{C}[z,w]$ defines a distinguished variety $\V= \mathcal Z_p
\cap \mathbb D^2$, where $\mathcal Z_p$ is the zero set of $p$.  Let $R$ be the
Riemann surface desingularizing $\mathcal Z_p$, with map $h:R\to \mathcal Z_p$.  Let
$S\subset R$ be the bordered Riemann surface $h^{-1}(V)$, so $h:S\to
V$ is a {\em holomap} in the sense of \cite{AM07}.  If $W$ is any
surface or variety, write $O(W)$ for the holomorphic functions on $W$.
(In particular, we recall that to say $f$ is holomorphic at $(z,w)\in
\V$ means ``$f$ extends to be holomorphic in a neighborhood of $(z,w)$
in $\mathbb C^2$.'')  If $W$ is a surface or variety with (always
assumed smooth) boundary $\partial W$, then $A(W)$ denotes those
functions continuous on $\overline W= W\cup \partial W$ and
holomorphic on $W$.  Finally, $H^\infty(W)$ denotes the algebra of
bounded analytic functions on $W$.  We remark that if $W$ is a Riemann
surface with smooth boundary, and $\omega$ is harmonic measure on
$\partial W$, then $H^\infty(W)$ coincides with $H^\infty(\omega)$ as
defined in the theory of uniform algebras (as the weak-* closure of
$A(W)$ in $L^\infty(\omega)$).  This precise result is found in
\cite{Gamelin-Lumer} Theorem 3.10, page 171.

%%??ref for this?  Wermer?
%%%Ahern-Sarason?

We recall some terminology and a theorem from \cite{AM07}.  
\begin{definition}If $S$ is a bordered Riemann surface, a linear functional on $O(S)$ is called {\em local} if it comes from a finitely supported distribution, i.e. has the form
\begin{equation*}
\Lambda(f)=\sum_{i=1}^m\sum_{j=0}^n c_{ij} f^{(j)}(\alpha_i).
\end{equation*}  
It is assumed that for each $i$, some $c_{ij}\neq 0$.  The set $\{\alpha_1,\dots \alpha_m\}$ is then called the {\em support} of $\Lambda$.  

A {\em connection} $\Gamma$ supported in $\{\alpha_1,\dots \alpha_m\}$ is a finite set of local functionals $\Lambda$ supported in $\{\alpha_1,\dots \alpha_m\}$.  Write $\displaystyle{\Gamma^\bot=\bigcap_{\Lambda\in \Gamma} {\rm ker}\Lambda}$.  Say $\Gamma$ is {\em algebraic} if $\Gamma^\bot $ is an algebra, and {\em irreducible} if every $f\in \Gamma^\bot$ is constant on the support of $\Gamma$.  
\end{definition}
A theorem of Gamelin \cite{G2} says that the finite codimension subalgebras of $O(S)$ are exactly the $\Gamma^\bot$'s for algebraic connections $\Gamma$.  Moreover, each connection is the union of finitely many irreducible connections with disjoint supports.  Finally, each finite codimension subalgebra $A\subset O(S)$ has a filtration $A_n\subsetneq A_{n-1}\cdots \subsetneq A_1=O(S)$ where each $A_j$ has codimension $1$ in the next and $A_{j+1}$ is obtained either as the kernel of a point derivation on $A_j$ or by identifying two points of the maximal ideal space of $A_j$. 

The main step in our proof will be an appeal to the following, which
is Theorem 2.8(i) of \cite{AM07}.  For us, $V$ will always be the
intersection of $\mathcal Z_p$ with a bidisk $U$ centered at $(0,0)$ (of some
radius) and $S$ will always be the piece of the disingularization
living over $V$. Note that an algebraic curve intersected with a
bounded domain in $\mathbb{C}^n$ is what is called a \emph{hyperbolic
  algebraic curve} and this is a special case of the hyperbolic
analytic curves defined in \cite{AM07}.
 
\begin{theorem}\label{thm:cofinite}
If $h:S\to V\subset U$ is a holomap from a Riemann surface $S$ onto a hyperbolic analytic curve $V\subset U$, then 
\[
A_h := \{F\circ h: F\in O(V)\}
\]
is a finite codimension subalgebra of $O(S)$.  
\end{theorem}

%We can now state the approximation theorem for bounded holomorphic functions on $V$.  
We can now prove the approximation theorem.

\begin{proof}[Proof of Theorem~\ref{thm:poly-approx}]  
We allow that $\V$ may have singularities on $\mathbb T^2$.  First we
extend $\V$ slightly: choose $r>1$ so that $\V_r:=\mathcal Z_p\cap r\mathbb
D^2$ has no additional singularities.  Let $S_r$ be the piece of the
desingularization lying over $\V_r$.  Then $S_r$ is also a bordered
Riemann surface, and $\overline {S}$ is compactly contained in $S_r$.
From the theory of hypo-Dirichlet algebras \cite{G1}, every function
in $H^\infty(S)$ can be contractively locally uniformly approximated
on $S$ by functions in $O(S_r)$.  (In particular, from \cite[Theorem
  IV.8.1]{G1} every function in $A(S)$ can be uniformly approximated
on $\overline{S}$ by functions in $O(S_r)$, and from \cite[Theorem
  VI.5.2]{G1} each $f\in H^\infty(S)$ can be approximated pointwise on
$S$ (and hence locally uniformly) with functions $f_n\in A(S)$,
satisfying $\|f_n\|_S\leq \|f\|_S$.)

Fix the function $f\in H^\infty(\V)$ that we would like to approximate
with polynomials. We may assume $\|f\|_\infty=1$.  Then $f\circ h$
belongs to $H^\infty(S)$, and so $f\circ h$ is approximated on
$\overline S$ by functions which extend to be holomorphic on $S_r$.
On the other hand, let $O_h(S_r)$ denote the subalgebra of functions
$\{F\circ h:F\in O(\V_r)\}$; by Theorem~\ref{thm:cofinite}, this is a
finite codimension subalgebra of $O(S_r)$, and hence by Gamelin's theorem is of
the form $\Gamma^\bot $ for some connection $\Gamma$ on $S_r$.  The
idea of the proof is to ``correct'' the approximants from $O(S_r)$ so
that they belong to $O_h(S_r)$.  It then follows from
Theorem~\ref{thm:cofinite} that the corrected approximants can be
pushed down to holomorphic functions on $\V_r$.  This process is
straightforward for the portion of $\Gamma$ supported in the interior
of $S$, but when the support of $\Gamma$ meets $\partial S$, it seems
that some care is needed (this is the case when $\V$ has singularities
on its boundary in $\mathbb T^2$).  For the Neil parabola and the annulus
 discussed in Section \ref{sec:examples} there are no singularities on
 the boundary which explains why it is possible to give simple 
 proofs that $H^\infty_{\mathcal K}$ and $H^\infty$ are isometric
 in these cases.  On the other hand, when a triply connected domain
 is realized as a distinguished variety, there are singularities on the
  boundary \cite{Ru}. 

Consider a sequence $(q_n)\subset O(S_r)$ converging uniformly to
$f\circ h$ on compact subsets of $S$, with each $q_n$ bounded by $1$
on $S$.  By Gamelin's theorem, $O_h(S_r)=\Gamma^\bot$ for some
algebraic connection $\Gamma$.  Since $\V$ meets the boundary of
$\mathbb D^2$ only in $\mathbb T^2$, it follows that each irreducible
component of $\Gamma$ is supported either entirely in the interior of
$S$ or entirely in the boundary of $S$ (points in the interior of $S$
cannot be identified with points in the boundary of $S$ when we push
forward to $\V$).  Decompose $\Gamma=\Gamma_1\cup \Gamma_2$ into its
interior and boundary pieces.  We first correct the $q_n$ to lie in
$\Gamma_1^\bot$, then correct these functions to lie in
$\Gamma_2^\bot$ as well.

Let 
\[
\Gamma_1^\bot := A_m\subsetneq A_{m-1} \cdots \subsetneq A_1=O(S_r)
\]
be a Gamelin filtration.  We show by induction that for each
$k=1,\dots m$ there exists a sequence $(q_n^k)\subset A_k$
approximating $f$ in the required way.  We already have $q_n^1=q_n$.
Suppose $(q_n^k)$ is given.  Now $A_{k+1}$ is obtained from $A_k$ as
$A_{k+1}=\ker \gamma_{k+1}$, where $\gamma_{k+1}$ is either a point
derivation or identifies two points.  In either case, choose $a_k\in
A_k$ such that $\gamma_{k+1}(a_k)=1$.  Define
\[
q_n^{k+1} = q_n^k -\gamma_{k+1}(q_n^k)a_k.
\]
By construction, $q_n^{k+1}$ lies in $A_{k+1}$ and converges locally
uniformly to $f$; since $\gamma_{k+1}(q_n^k)\to \gamma_{k+1}(f)=0$,
the sup norms of the $q_n^{k+1}$ converge to $1$, so after
normalization the $q_n^{k+1}$ work.

%First, suppose $A_{k+1}$ is obtained from $A_k$ by identifying two points.  Since $S$ is the desingularization of a variety, these points either both lie %in the interior, or are both on the boundary.  In the interior case, supposing $a$ and $b$ are the points identified, choose $q\in A_k$ with $q(a)=0, q(b)%=1$.  Then define
%\[
%q_n^{k+1} = q_n^k -[q_n^k(b)-q_n^k(a)]q
%\]
%Since $f\circ h$ identifies $a$ and $b$ and $q_n^k\to f$ in the interior, we have $(q_n^k(a)-q_n^k(b)q\to 0$ uniformly on $V$, so (after renormalizing) $q%_n^{k+1}$ is the desired approximating sequence in $A_{k+1}$.  It is evident that the analogous construction works for a point derivation at an interior p%oint:  take $q\in A_k$ with $q^{(N)}(a)=1$, and put
%\[
%q_n^{k+1} = q_n^k - (q_n^k)^{(N)}(a) q
%\]
%Again since $q_n^k\to f$ uniformly near $a$, the derivatives converge uniformly as well, and since the derivative of $f$ vanishes to order at least $N$ at% $a$, we have $(q_n^k)^{(N)}(a) q\to 0$ uniformly on $V$, so again these $q_n^{k+1}$ work after normalization.  

%So far, the $q_n$ approximate $f$ and satisfy the interior relations; we now modify them to also satisfy the boundary relations.  Once this is done they c%an be pushed down to $V$. 

To accomplish the modification on the boundary, we multiply the
functions $q_n$ by functions $G_n$ that converge to $1$ pointwise in
$S$ and ``zero out'' the boundary relations.  The $G_n$ are
constructed using two lemmas:

\begin{lemma}
Let $S, \V, \Gamma$ be as above.  Let $\alpha_1, \dots \alpha_m$ be
the interior points of $S$ belonging to the support of $\Gamma$, let
$\beta_1, \dots \beta_l$ be the boundary points in the support of
$\Gamma$, and let an integer $N\geq 1$ be given.  Then there exists a
function $b$, holomorphic in a neighborhood of $\overline{S}$, such
that
\begin{itemize}
\item[i)] $b$ is inner (that is, $|b|=1$ on $\partial S$),
\item[ii)]  $b$ vanishes to order $N$ at each $\alpha_j$, and
\item[iii)] $b$ is $1$ at each $\beta_j$.
\end{itemize}
\end{lemma}
\begin{proof}
Write $h=(h_1,h_2)$ and consider the projection $h_1:\overline{S_r}\to
\overline {r\mathbb D}$.  It is straightforward to construct a finite
Blaschke product $B$ which vanishes to order $N$ at each of the points
$h_1(\alpha_j)$, and takes the value $1$ at the points $h_1(\beta_j)$
on the unit circle.  By shrinking $r$ if necessary, $b=B\circ h_1$
does the job.

\end{proof}

\begin{lemma}  There exists a sequence of functions $g_n$ in the unit disk such that:
\begin{itemize}
\item[i)]  Each $g_n$ is holomorphic in some neighborhood of $\overline{\mathbb D}$ and bounded by $1$ in $\mathbb D$, 
\item[ii)]  $g_n(1)=0$ for all $n$, and
\item[iii)]  $g_n\to 1$ uniformly on compact subsets of $\mathbb D$.  
\end{itemize}

\end{lemma}

\begin{proof}[Proof of lemma]  To construct the $g_n$, let $c_n=1-n^{-2}$ and define
\[
h_n(z)= \exp\left( -\frac1n \left(\frac{1+c_nz}{1-c_nz}\right)\right)-\exp\left(-\frac1n \left(\frac{1+c_n}{1-c_n}\right)\right)
\]
It is evident that $h_n$ is holomorphic on $\overline{\mathbb D}$ and
that $h_n(1)=0$ for all $n$.  Moreover it is readily verified that
$\|h_n\|_\infty\leq 1+o(1)$ as $n\to \infty$ and $h_n\to 1$ locally
uniformly in $\mathbb D$.  Taking $g_n=h_n/\|h_n\|_{\infty}$ works.
\end{proof}
Now we combine the two lemmas.  For each $n$, we may shrink the domain
of $b$ further (but so that it still contains $\overline{S}$) so that
$b$ maps into the domain of $g_n$.  We may then form the composition
$g_n\circ b$.  Now, $b$ is bounded by $1$ in $\overline S$, and $g_n$
is given by a uniformly convergent power series on $\mathbb
\overline{D}$, and vanishes at each $\beta_j$.  So by taking a
suitably high power $G_n =(g_n\circ b)^N$, we see that each $G_n$ is
annihilated by $\Gamma$ (it vanishes to high order at the boundary
points, and satisfies the interior relations because $b$ does).  Thus,
if we call $S_n$ the domain of $G_n$, then each $G_n$ belongs to
$O_h(S_n)$.  By construction the $G_n$ are all bounded by $1$ in $S$,
and $G_n\to 1$ pointwise on $S$.

We can now use the $G_n$ to correct the sequence $q_n$ converging to
$f\circ h$.  In particular, by construction the product $G_nq_n$
belongs to $O_h(S_n)$, since the relations on the boundary are
zeroed out by the $G_n$.  Setting $\V_n=h(S_n)$, from
Theorem~\ref{thm:cofinite} there is an analytic function $p_n$ on
$\V_n$ satisfying $p_n\circ h=G_nq_n$.  So the $p_n$ are each
holomorphic on a neighborhood of $\V$ in $\mathbb C^2$, bounded by $1$
on $\V$, and converge to $f$ uniformly on compact subsets of $\V$.
Finally, since $\overline{\V}$ is polynomially convex, the Oka-Weil
theorem says that each $p_n$ is uniformly approximable on
$\overline{\V}$ by polynomials, and thus $f$ is approximable by
polynomials as desired.
\end{proof}

\section{Bounded Analytic Functions on $\V$}
 \label{sec:analytic-functions}
In this section we prove Corollary~\ref{cor:isometry-of-algebras}. By
Corollary~\ref{cor:polys_are_multipliers}, every polynomial belongs to
$H^\infty_{\mathcal K}(\V)$, with norm equal to the supremum norm over
$\V$.  The first step is an elementary completeness result for $H^\infty_{\cK}(\V)$.
%For completeness, we also note, as is usual for spaces of multipliers,
%that $H^\infty_{\mathcal K}(\V)$ is complete (with respect to the norm
%$\|f\|_{\V}$) and closed with respect to bounded pointwise
%convergence.  This fact is proved first.
\begin{proposition}
  \label{props-of-H-infty}
    The algebra $H^\infty_{\mathcal K}(\V)$ is closed both in norm
    and under pointwise bounded convergence.
\end{proposition}

  Since the result is standard (see for instance  
  \cite{AMbook}), we only sketch a proof.

\begin{proof}
  Let $(f_n)$ be a given sequence from $H^\infty_{\mathcal K}(\V)$ and suppose
  there is a $C$ such that $\|f_n\|\le C$ independent of $n$. Further,
  assume $f_n$ converges pointwise on $\V$.  It follows that for every
  finite subset $F$ of $\V$, every admissible kernel $K$ and every
  $n$, the (block) matrix
 \[
  \begin{pmatrix} (C^2-f_n(x)\overline{f_n(y)})K(x,y)\end{pmatrix}_{x,y\in F}
 \] 
   is positive semidefinite.  Thus, 
 \[
   \begin{pmatrix} (C^2-f(x)\overline{f(x)})K(x,y) \end{pmatrix}_{x,y\in F}
 \] 
  is positive semi-definite and hence $f\in H^\infty_{\mathcal K}(\V)$.

  Now suppose $(f_n)$ is Cauchy in $H^\infty_{\mathcal K}(\V)$. 
  Since $\|f\|_{\V}$ dominates $\|f\|_\infty$, the 
  sequence converges pointwise to some $f$. It
  follows that $f\in H^\infty_{\mathcal K}(\V)$ and 
  moreover $\|f\|\le C$.  It remains to
  verify that $(f_n)$ converges to $f$ in $H^\infty_{\mathcal K}(\V)$.

  Let $\epsilon >0$ be given.  There is an $N$ so that 
  if $m,n\ge N$, then $\|f_m-f_n\|_{\V} <\epsilon$.
  From what has already been proved, it now follows that
 \[
   \|f-f_n\|_{\V}\le \epsilon.
 \]
\end{proof}

%Recall (from the discussion following Corollary 1.10) that
%$H^\infty(\V)$ denotes the algebra of all bounded analytic functions
%on $\V$, equipped with the supremum norm $\|f\|_\infty:=\sup_{z\in\V}
%|f(z)|$.  We will write $P^\infty(\V)$ for the subalgebra consisting
%of the closure of the polynomials in the topology of bounded pointwise
%convergence; that is, a function $f$ on $\V$ belongs to $P^\infty(\V)$
%if and only if there exists a uniformly bounded sequence of
%polynomials converging to $f$ pointwise on $\V$.  $P^\infty(\V)$ is of
%course also equipped with the supremum norm.  It may not be
%immediately obvious that elements of $P^\infty(\V)$ are analytic on
%$\V$, but this will follow from Corollary~\ref{cor:isometry-of-algebras}; hence
%$P^\infty(\V)$ is a subalgebra of $H^\infty(\V)$.

\begin{proof}[Proof of Corollary~\ref{cor:isometry-of-algebras}]  
Suppose $f\in H^\infty(\V)$.  Then by Theorem~\ref{thm:poly-approx}
there exist polynomials $p_n\to f$ pointwise with $\|p_n\|_\infty\leq
\|f\|_\infty$.  By Corollary~\ref{cor:polys_are_multipliers}, each
$p_n$ belongs to $H^\infty_{\mathcal K}(\V)$, and
$\|p_n\|_\infty=\|p_n\|_{\V}$.  It follows that $f\in
H^\infty_{\mathcal K}(\V)$ and $\|f\|_{\V}\leq \|f\|_\infty$ by
Proposition~\ref{props-of-H-infty}.

We now turn to the proof that each function in $H^\infty_{\mathcal
  K}(\V)$ is analytic on $\V$.  It is proved in \cite{Kn2} (Theorem
11.3) that every distinguished variety has an admissible kernel
\[
K((z,w),(\zeta,\eta)) = \frac{Q(z,w)Q(\zeta,\eta)^*}{1-z\zetastar}
\]
where $K((z,w),(z,w)) \ne 0$ for all $(z,w) \in \V$.  Indeed, it is
shown that $Q$ can be chosen to be of the form 
\[
(1,w,\dots, w^{m-1}) A(z)
\]
where $A(z)$ is an $m\times m$ matrix polynomial which is invertible
for every $z$ in $\mathbb{D}$.  Let $f$ belong to the unit ball of $H^\infty_{\mathcal K}(\V)$.
Then the kernel
\begin{equation*}
\left(1-f(z,w)\overline{f(\zeta,\eta)}\right) 
  \frac{Q(z,w)Q(\zeta, \eta)^*}{1-z\overline{\zeta}}
\end{equation*}
 is positive, and hence there exists a (vector-valued) function $\Gamma$ 
 on $\V$ such that 
\begin{equation*}
(1-f(z,w)\overline{f(\zeta, \eta)})Q(z,w)Q(\zeta,\eta)^* = \Gamma(z,w)(1-z\overline{\zeta}) \Gamma(\zeta,\eta)^*.
\end{equation*}
 A straightforward lurking isometry argument
 produces a contractive $m\times m$ $H^\infty(\mathbb{D})$ matrix 
 function $F$ such that
\begin{equation}\label{eqn:F_eigen}
  F(z)^* Q(z,w)^*  = \overline{f(z,w)} Q(z,w)^*
\end{equation}
  for all $(z,w)\in \V$.  Since $K$ (hence $Q$) does not vanish at
 $(z_0, w_0)$, some coordinate of $Q$ doesn't vanish in a neighborhood
 of $(z_0,w_0)$, say $q_j$.  Writing out the $j^{th}$ coordinate of
 (\ref{eqn:F_eigen}) and taking conjugates gives
\begin{equation*}
  f(z,w) = \frac{\sum_{i=1}^m q_i(z,w) F_{ij}(z)}{q_j(z,w)}.
\end{equation*}
 The right-hand side extends to be analytic in a neighborhood of
 $(z_0, w_0)$ in $\mathbb D^2$, hence $f$ is holomorphic (as a
 function on $\V$) at $(z_0, w_0)$.  Finally, as already noted, the
 inequality $\|f\|_{\infty}\leq\|f\|_{\V}$ is trivial.
\end{proof}

%Corollary~\ref{cor:isometry-of-algebras} leads naturally to the following
%questions, which we have been unable to answer:

%{\bf Question:} {\em Are either (or both) of the inclusions in
%Corollary~\ref{cor:isometry-of-algebras} equalities?  Isometries?}

%Propositions~\ref{prop:Neil}, \ref{prop:annulus} and
%\ref{prop:higher_connectivity} show that the answer to both of these
%questions is ``yes'' for the Neil parabola $\mathcal N$, and for
%doubly or triply connected planar domains $R$ (viewed as distinguished
%varieties).  We anticipate the same is true for domains of higher
%connectivity, and as noted in the introduction, we know no examples of
%$\V$ for which $P^\infty(\mathcal V)\neq H^\infty(\mathcal V)$.

%?? Will the $F$ coming from the lurking isometry in the proof
%  be related to $f(z,\Phi(z))$?? 

\end{document}